\documentclass{svjour3}  
\usepackage{times} 
\usepackage{amsmath,lipsum} 
\usepackage{amssymb,amsfonts}  
\usepackage{latexsym,theorem,epsfig}
\usepackage{graphicx} 
\usepackage{dsfont}
\usepackage{mathtools}
\usepackage{subfigure}
\usepackage{lscape}
\usepackage{color}
\usepackage{float}
\usepackage{enumerate}
\usepackage[usenames,dvipsnames]{pstricks}
\usepackage{epsfig}
\usepackage{pst-grad} 
\usepackage{pst-plot} 
\usepackage{mathrsfs}
\usepackage{extarrows}
\usepackage[percent]{overpic}
\usepackage[numbers]{natbib}
\usepackage[misc]{ifsym}
\usepackage{caption}

\newtheorem{assumption}{Assumption}
\allowdisplaybreaks

\DeclareMathOperator*{\argmin}{arg\,min}
\title{
Distributed Learning for Cooperative Inference}

\author{Angelia Nedi\'{c} \and Alex Olshevsky \and C\'{e}sar A.\ Uribe$^\dagger$\thanks{$\dagger$This research is supported partially by the National Science Foundation under
		grant no.\  CPS 15-44953 and by the Office of Naval Research under grant
		no.\ N00014-17-1-2195. A preliminary version of some results presented in this paper were  published in the IEEE Conference of Decision and Control 2016 \cite{ned16d}. } }

\institute{A.\ Nedi\'{c} \at
	ECEE Department, Arizona State University \\
	\email{angelia.nedich@asu.edu}           
	\and
	A.\ Olshevsky \at
	Department of ECE and Division of Systems Engineering, Boston University \\
	\email{alexols@bu.edu}  
	\and 
	C.A.\ Uribe \Letter\at
	ECE Department and Coordinated Science Laboratory, University of Illinois \\
	\email{cauribe2@illinois.edu}  
}

\begin{document}
\maketitle
\begin{abstract}
We study the problem of cooperative inference where a group of agents  interact over a network and seek to estimate a joint parameter that best explains a set of observations. Agents do not know the network topology or the observations of other agents. We explore a variational interpretation of the Bayesian posterior density, and its relation to the stochastic mirror descent algorithm, to propose a new distributed learning algorithm. We show that, under appropriate assumptions, the beliefs generated by the proposed algorithm concentrate around the true parameter exponentially fast. We provide explicit non-asymptotic bounds for the convergence rate. Moreover, we develop explicit and computationally efficient algorithms for observation models belonging to exponential families.
\end{abstract}

\section{Introduction}

The increasing amount of data generated by recent applications of distributed systems such as social media, sensor networks, and cloud-based databases has brought considerable attention to distributed data processing approaches, in particular the design of distributed algorithms that take into account the communication constraints and make coordinated decisions in a distributed manner~\cite{jad12,rah10,ala04,olf06,aum76,bor82,tsi84,gen86,coo90,deg74,gil93}. In a distributed system, the interactions between agents are usually restricted to follow certain constraints on the flow of information imposed by the network structure. Such information constraints cause the agents to only be able to use locally available information. This contrasts with centralized approaches where all information and computation resources are available at a single location \cite{gub93,zhu05,vis97,sun04}.  
 
One traditional problem in decision-making is that of parameter estimation or statistical learning. Given a set of noisy observations coming from a joint distribution one would like to estimate a parameter or distribution that minimizes a certain loss function. For example, Maximum a Posteriori (MAP) or 
Minimum Least Squared Error (MLSE) estimators fit a parameter to some model of the observations. Both, MAP and MLSE estimators require some form of Bayesian posterior computation based on models that explain the observations for a given parameter. Computation of such a posteriori distributions depends on having exact models about the likelihood of the corresponding observations. This is  one of the main difficulties of using Bayesian approaches in a distributed setting.  A fully Bayesian approach is not possible because full knowledge of the network structure, or of other agents' likelihood models, may not be available~\cite{gal03,mos10,ace11}.
 
Following the seminal work of Jadbabaie et al.\ in \cite{jad12,jad13,sha13}, there have been many studies of distributed non-Bayesian update rules over networks. In this case, agents are assumed to be boundedly rational (i.e., they fail to aggregate information in a fully Bayesian way \cite{gol10}). Proposed non-Bayesian algorithms involve an aggregation step, typically consisting of weighted geometric or arithmetic average of the received beliefs~\cite{ace08,tsi84,jad03,ned13,ols14}, and a Bayesian update with the locally available data~\cite{ace11,mos14}. 
Recent studies proposed variations of the non-Bayesian approach and proved consistent, geometric and non-asymptotic convergence rates for a  general class of distributed algorithms; from asymptotic analysis \cite{sha13,lal14,qip11,qip15,sha15,rah15} to non-asymptotic bounds \cite{sha14,ned15,lal14b,ned14}, time-varying directed graphs \cite{ned15b}, and transmission and node failures \cite{su16}; see \cite{bar13,ned16c} for an extended literature review.  
  
We build upon the work in~\cite{bir15} on non-asymptotic behaviors of Bayesian estimators to derive new non-asymptotic concentration results for distributed learning algorithms. In contrast to the existing results which assume a finite hypothesis set, in this paper we extend the framework to countably many and a continuum of hypotheses. Our results show that in general, the network structure will induce a transient time after which all agents learn at a network independent rate, and this rate is geometric. 
 
The contributions of this paper are as follows. We begin with a variational analysis of Bayesian posterior and derive an optimization problem for which the posterior is a step of the Stochastic Mirror Descent method. We then use this interpretation to propose a distributed Stochastic Mirror Descent method for distributed learning.  We show that this distributed learning algorithm concentrates the beliefs of all agents around the true parameter at an exponential rate. We derive high probability non-asymptotic bounds for the convergence rate. In contrast to the existing literature, we analyze the case where the parameter spaces are compact. Moreover, we specialize the proposed algorithm to parametric models of an exponential family which results in especially simple updates.
 
The rest of this paper is organized as follows. Section \ref{sec:setup} introduces the problem setup, it describes the networked observation model and the inference task. Section \ref{sec:variational} presents a variational analysis of the Bayesian posterior, shows the implicit representation of the posterior as steps in a stochastic program and extends this program to the distributed setup. Section \ref{sec:inference} specializes the proposed distributed learning protocol to the case of observation models that are members of the exponential family. Section \ref{sec:concentration} shows our main results about the exponential concentration of beliefs around the true parameter.
Section \ref{sec:concentration} begins by gently introducing our techniques by proving a concentration result in the case of countably many hypotheses, before turning to our main focus: the case when the set of hypotheses is a compact subset of $\mathbb{R}^d$.  Finally, conclusions, open problems, and potential future work are discussed.

\textbf{\textit{Notation}}: 
Random variables are denoted with upper-case letters, e.g. $X$, 
while the corresponding lower-case are used for their realizations, e.g. $x$. 
Time indices are denoted by subscripts, and the letter $k$ or $t$ is generally used. 
Agent indices are denoted by superscripts, and the letters $i$ or $j$ are used. 
We write $[A]_{ij}$ or $a_{ij}$ to denote the entry of a matrix $A$ in its $i$-th row and $j$-th column. 
We use $A'$ for the transpose of a matrix $A$, and $x'$ for the transpose of a vector $x$. 
The complement of a set $B$ is denoted as $B^c$. 
 
 \section{Problem Setup}\label{sec:setup}
 
We begin by introducing the learning problem from a centralized perspective, where all information is available at a single location. Later, we will generalize the setup to the distributed setting where only partial and distributed information is available. 
 
 Consider a probability space $(\Omega,\mathcal{F},\mathbb{P})$, where $\Omega$ is a sample space, $\mathcal{F}$ is a $\sigma$-algebra and $\mathbb{P}$ a probability measure. Assume that we observe a sequence of independent random variables $X_1,X_2,\hdots$, all taking values in some measurable space 
 $(\mathcal{X},\mathcal{A})$ and identically distributed with a common \textit{unknown} distribution $P$. In addition, we have a parametrized family of distributions ${{\mathscr{P} = \{P_{\theta} : \theta \in \Theta\}}}$,where the map $\Theta \to \mathscr{P}$ from parameter to distribution is one-to-one. Moreover, 
 the models in $\mathscr{P}$ are all dominated\footnote{A measure $\mu$ is dominated by (or absolutely continuous with respect to) a measure $\lambda$ if $\lambda(B) = 0$ implies $\mu(B)=0$ for every measurable set~$B$.} by a $\sigma$-finite measure $\lambda$, with corresponding densities $p_\theta = dP_\theta / d\lambda$. Assuming that there exists a $\theta^*$ such that $P_{\theta^*} = P$, the objective is to estimate $\theta^*$ based on the received observations $x_1,x_2,\hdots$.
 
 Following a Bayesian approach, we begin with a prior on  $\theta^*$ represented as a distribution on the space $\Theta$; then given a sequence of observations, we incorporate such knowledge into a posterior distribution following Bayes' rule. Specifically, we assume that $\Theta$ is equipped with a $\sigma$-algebra and a measure $\sigma$ and that $\mu_0$, which is our prior belief, is a probability measure on $\Theta$ which is dominated by $\sigma$. Furthermore, the densities $p_{\theta}(x)$ are measurable functions of $\theta$ for any 
 $x \in \mathcal{X}$, and also dominated by $\sigma$. We then define the belief $\mu_k$  as the posterior distribution given the sequence of observations up to time $k$, i.e.,
 \begin{align}\label{bayes}
 \mu_{k+1}(B) & \propto \int_{B} \prod\limits_{t=1}^{k+1} p_\theta(x_{t}) d\mu_0(\theta) .
 \end{align} for any measurable set $B \subset \Theta$ (note that we used the independence of the observations at each time step).
Assuming that all observations are readily available at a centralized location, under appropriate conditions, the recursive Bayesian posterior in Eq.~\eqref{bayes} will be consistent in the sense that the beliefs $\mu_k$ will concentrate around $\theta^*$; see \cite{gho97,sch65,gho00} for a formal statement. Several authors have studied the rate at which this concentration occurs, in both asymptotic and non-asymptotic regimes \cite{bir15,gho07,riv12}. 
 
Now consider the case where there is a network of $n$ agents observing the process $X_1,X_2,\hdots$, where $X_k$ is now a random vector belonging to the product space $\prod_{i=1}^{n}\mathcal{X}^i$, and  
$X_k = [X^1_k,X^2_k,\hdots,X^n_k ]'$ consists of observations $X_k^i$ of the agents at time~$k$. 
Specifically, agent $i$  observes the sequence $X_1^i,X_2^i,\hdots$,  
where $X_k^i$ is now distributed according to an unknown distributions $P^i$. 
Each agent agent $i$ has a private family of distributions ${{\mathscr{P}^i = \{P_{\theta}^i : \theta \in \Theta\}}}$ it would like to fit to the observations.
However, the goal is for {\em all} agents to agree on a {\em single} $\theta$ that best explains the complete set of observations. In other words, the agents collaboratively seek to find a $\theta^*$ that makes the distribution $\boldsymbol{P}_{\theta^*} = \prod_{i=1}^{n} P^i_{\theta^*} $ as close as possible to the unknown true distribution $P = \prod_{i=1}^n P^i$.  Agents interact over a network defined by an undirected graph $\mathcal{G} = (V,E)$, where $V = \{1,2,\ldots,n\}$ is the set of agents and $E$ is a set of undirected edges,
i.e., $\left(i,j\right) \in E$ if and only if agents $i$ and $j$ can communicate with each other. 

We study a simple interaction model where, at each step, agents exchange their beliefs with their neighbors in the graph. Thus at every time step $k$,  agent $i$ will receive the sample $x_{k}^i$ from $X_k^i$ as well as the beliefs of its neighboring agents, i.e., it will receive $\mu_{k-1}^j$ for all $j$ such that $(i,j) \in E$.  Applying a fully Bayesian approach runs into some obstacles in this setting, as agents know neither the network topology nor the private family of distributions of other agents.
Our goal is to design a learning procedure which is both distributed and consistent. That is, we are interested in a belief update algorithm that aggregates information in a non-Bayesian manner and guarantees that the beliefs of all agents will concentrate around $\theta^*$. 
 
As a motivating example,  consider the problem of distributed source localization \cite{rab04,rab05}. In this scenario, a network of $n$ agents receives noisy measurements of the distance to a source. The sensing capabilities of each sensor might be limited to a certain region. 
The group objective is to jointly identify the location of the source. Figure \ref{location} shows a group of $7$ agents (circles) seeking to localize a source (star). 
There is an underlying graph that indicates which nodes can exchange messages.  
Moreover, each node has a sensing region indicated by the dashed circle around it. 
Each agent observes signals proportional to the distance to the target. Since a target cannot be localized effectively from a single measure of the distance, agents must cooperate to have any hope of achieving decent localization. For more details on the problem, as well as simulations of the several discrete learning rules, 
we refer the reader to our earlier paper~\cite{ned15} dealing with the case when the set $\Theta$ is finite.

\begin{figure}[h]
	\centering
	\includegraphics[width=0.3\textwidth]{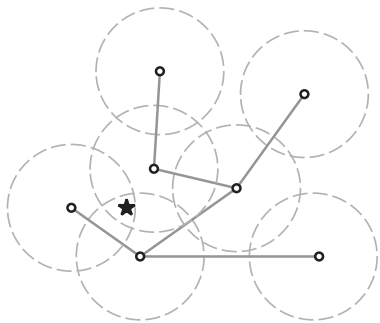}
	\caption{Distributed source localization example.}
	\label{location}
\end{figure} 

  \section{A variational approach to distributed Bayesian filtering}\label{sec:variational}

In this section, we make the observation that the posterior in Eq.~\eqref{bayes} corresponds to an iteration of a first-order optimization algorithm, namely Stochastic Mirror Descent~\cite{bec03,ned14b,dai15,rab15}. Closely related variational interpretations of Bayes' rule are well-known, and in particular have been given in  \cite{zel88,wal06,hil12}.  The specific connection to Stochastic Mirror Descent has not been noted, as far as we are aware of.  This connection will serve to motivate a distributed learning method which will be the main focus of the paper.

\vspace{-0.4cm}
\subsection{Bayes' rule as Stochastic Mirror Descent} Suppose we want to solve the following optimization problem 
\begin{align}\label{central_dk}
\min_{\theta \in \Theta} F(\theta) & =  D_{KL}(P\|P_\theta),
\end{align}
where $P$ is an unknown true distribution and $P_{\theta}$ is a parametrized family of distributions (see 
Section~\ref{sec:setup}). Here, $D_{KL}(P\|Q)$ is the Kullback-Leibler (KL) divergence\footnote{$D_{KL}(P\|Q)$  between distributions $P$ and $Q$ (with $P$ dominated by $Q$)  is defined to be
	$D_{KL}(P\|Q)  = - \mathbb{E}_P \left[\log dQ / dP\right].$
	} between distributions $P$ and $Q$. 

First note that we can rewrite Eq.~\eqref{central_dk} as
\begin{align*}
	\min_{\theta \in \Theta} D_{KL}(P\|P_\theta) & = \min_{\pi \in \Delta_{\Theta} } \mathbb{E}_{\pi} D_{KL}(P\|P_\theta)  \text{ \qquad   s.t. } \theta\sim \pi \nonumber \\
	& = \min_{\pi \in \Delta_{\Theta} } \mathbb{E}_{\pi} \mathbb{E}_{P} \left[-\log \frac{dP_{\theta}}{dP}\right],
	\end{align*}
	where $\Delta_{\Theta}$ is the set of all possible densities on the parameter space $\Theta$. Since the distribution $P$ does not depend on the parameter $\theta$, it follows that
\begin{align}\label{central_expectation}	
	\argmin_{\theta \in \Theta} D_{KL}(P\|P_\theta)& = \argmin_{\pi \in \Delta_{\Theta} } \mathbb{E}_{\pi} \mathbb{E}_{P} \left[-\log p_{\theta}(X)\right] \text{ \  where \ } \theta\sim \pi \text{\  and \ } X \sim P \nonumber \\
	& = \argmin_{\pi \in \Delta_{\Theta} } \mathbb{E}_{P} \mathbb{E}_{\pi} \left[-\log p_{\theta}(X)\right] \text{ \  where \ } \theta\sim \pi \text{\  and \ } X \sim P .
\end{align}

 The equality in Eq.~\eqref{central_expectation}, where we exchange the order of the expectations, follows from the Fubini-Tonelli theorem. Clearly, if $\theta^*$ minimizes Eq.~\eqref{central_dk}, then a distributions which puts all the mass on $\theta^*$ minimizes  Eq.~\eqref{central_expectation}.

The difficulty in evaluating the objective function in Eq.~\eqref{central_expectation} lies in the fact that the distribution $P$ is unknown. A generic approach to solving such problems is using algorithms from stochastic approximation methods, where the objective is minimized by constructing a sequence of gradient-based iterates whereby the true gradient of the objective (which is not available) is replaced with a gradient sample that is available at a given time.

A particular method that is relevant for the solution of stochastic programs of the form 
\begin{align*}
\min_{x\in Z} \mathbb{E}\left[F(x,\Xi)\right],
\end{align*} 
for some random variable $\Xi$ with unknown distribution, is  the \textit{stochastic mirror descent} method \cite{jud08,ned14b,bec03,lan12}. The stochastic mirror descent approach constructs a sequence $\{x_k\}$ 
as follows:
\begin{align*}
x_{k+1} & =  \argmin_{x \in Z} \left\lbrace \langle \nabla F(x,\xi_k), x\rangle 
+ \frac{1}{\alpha_{k}} D_w(x,x_k)\right\rbrace ,
\end{align*}
for a realization $\xi_k$ of $\Xi$. Here, $\alpha_k>0$ is the step-size,
 $\langle p, q \rangle = \int_{\Theta} p(\theta) q(\theta) d \sigma$,  and $D_w(x,x_k)$ is a Bregman distance function associated with a distance-generating function $w$, i.e., 
\begin{align*}
	D_w(x,z) = w(z)-w(x) - \delta w[z; x-z],
	\end{align*} where $ \delta w[z; x-z]$ is the Fr\'{e}chet derivative of $w$ at $z$ in the direction of $x-z$. 

For Eq.~\eqref{central_expectation}, Stochastic Mirror Descent generates a sequence of densities $\{d\mu_k\}$, as follows:
\begin{align}\label{central}
d\mu_{k+1} & =  \argmin_{\pi \in \Delta_{\Theta}} \left\lbrace \langle - \log p_\theta(x_{k+1}), \pi\rangle 
+ \frac{1}{\alpha_{k}} D_w(\pi,d\mu_k)\right\rbrace ,  \qquad \text{where } \theta \sim \pi.
\end{align} If we choose $w(x) = \int x \log x$ as the distance-generating function, then the corresponding Bregman distance is the Kullback-Leibler (KL) divergence $D_{KL}$. Additionally, by selecting $\alpha_k=1$, the solution to the optimization problem in Eq.~\eqref{central} can be computed explicitly, where for each $\theta \in \Theta$,
\begin{align*}
d\mu_{k+1}(\theta) &  \propto  p_\theta(x_{k+1}) d\mu_k(\theta),
\end{align*}
which is the particular definition for the posterior distribution according to Eq.~\eqref{bayes} 
(a formal proof of this assertion is a special case of Proposition~\ref{mirror} shown later in the paper). 

\subsection{Distributed Stochastic Mirror Descent} 
Now, consider the distributed problem where the network of agents want to collectively solve the following optimization problem
\begin{align}\label{opt_problem}
\min_{\theta \in \Theta} F(\theta) & \triangleq D_{KL}\left(\boldsymbol{P}\|\boldsymbol{P}_{\theta}\right)  = \sum\limits_{i=1}^n D_{KL}(P^i\|P^i_{\theta}) .
\end{align}

Recall that the distribution $\boldsymbol{P}$ is unknown (though, of course, agents gain information about it by observing samples from $X_1^i, X_2^i, \ldots$ and interacting with other agents) and that $\mathscr{P}^i$ containing all the distributions $P_{\theta}^i$ is a private family of distributions and  is only available to agent $i$. 

We propose the following algorithm as a distributed version of the stochastic mirror descent for the solution of problem Eq.~\eqref{opt_problem}:
\begin{align}\label{distributed}
d\mu_{k+1}^i & =  \argmin_{\pi \in \Delta_{\Theta}} \Big\{\langle - \log p_\theta^i(x_{k+1}^i), \pi\rangle 
+ \sum\limits_{j=1}^{n} a_{ij} D_{KL}(\pi\|d\mu_k^j) \Big\} \qquad \text{where } \theta \sim \pi,
\end{align}
with $a_{ij}>0$ denoting the weight that agent $i$ assigns to beliefs coming from its neighbor $j$. Specifically,
$a_{ij}>0$ if $(i,j)\in E$ or $j=i$, and $a_{ij}=0$ if $(i,j)\notin E$.
The optimization problem in Eq.~\eqref{opt_problem} has a closed form solution. 
In particular, the posterior density at each $\theta \in \Theta$ is given by
\begin{align*}
d\mu_{k+1}^i(\theta)  & \propto p_{\theta}^i(x^i_{k+1})  \prod_{j=1}^{n}(d\mu_{k}^j(\theta))^{a_{ij}}, 
\end{align*}
or equivalently, the belief on a measurable set $B$ of an agent $i$ at time $k+1$ is
\begin{align}\label{protocol}
 \mu_{k+1}^i(B) & \propto \int_{B}  p_\theta^i(x_{k+1}^i) \prod_{j=1}^{n} (d\mu_k^j(\theta))^{a_{ij}} .
\end{align}

We state the correctness of this claim in the following proposition. 

\begin{proposition}\label{mirror}
	The probability measure $\mu_{k+1}^i$ over the set $\Theta$ defined by the update protocol Eq.~\eqref{protocol} coincides, almost everywhere, with the update 
the distributed stochastic mirror descent algorithm applied to the optimization problem in Eq.~\eqref{opt_problem}.
\end{proposition}
\begin{proof}
	We need to show that the density $d\mu_{k+1}^i$ associated with the probability measure $\mu_{k+1}^i$ defined by
	Eq.~\eqref{protocol} minimizes the problem in Eq.~\eqref{distributed}.
	To do so,  let $G(\pi)$ be the objective function for the problem in Eq.~\eqref{distributed},  i.e.,
	\begin{align*}
	G(\pi) & = \langle - \log p_\theta^i(x_{k+1}^i), \pi\rangle 
	+ \sum\limits_{j=1}^{n} a_{ij} D_{KL}(\pi\|d\mu_k^j).
	\end{align*}
	Next, we add and subtract the KL divergence between $\pi$ and the density  $d{\mu}_{k+1}^i$ to obtain
	\begin{align*}
	G(\pi) & 
	= \langle - \log p_\theta^i(x_{k+1}^i), \pi\rangle 
	+ \sum\limits_{j=1}^{n} a_{ij} D_{KL}(\pi\|d\mu_k^j) - D_{KL}\left(\pi\|d{\mu}_{k+1}^i\right)  + D_{KL}\left(\pi\|d{\mu}_{k+1}^i\right)\\
	&= \langle - \log p_\theta^i(x_{k+1}^i), \pi\rangle 
	+ D_{KL}\left(\pi\|d{\mu}_{k+1}^i\right) + \sum\limits_{j=1}^{n}a_{ij} \mathbb{E}_\pi \log \frac{d{\mu}_{k+1}^i}{d{\mu}_{k}^j}. 
	\end{align*}
	
	Now, from Eq.~\eqref{protocol} it follows that
	\begin{align}\label{prop1}
	G(\pi) &
	= \langle - \log p_\theta^i(x_{k+1}^i), \pi\rangle 
	+ D_{KL}\left(\pi\|d{\mu}_{k+1}^i\right) +  \nonumber\\
	& \qquad\qquad\sum\limits_{j=1}^{n}a_{ij} \mathbb{E}_\pi \log \left(  \frac{1}{d{\mu}_{k}^j} \frac{1}{Z_{k+1}^i}\prod\limits_{l=1}^{n} \left(d{\mu}_{k}^l\right)^{a_{il}}p^i_{\theta}(x_{k+1}^i) \right) \nonumber  \\
	&
	= \langle - \log p_\theta^i(x_{k+1}^i), \pi\rangle 
	+ D_{KL}\left(\pi\|d{\mu}_{k+1}^i\right)  \nonumber \\
	& \qquad\qquad -  \log Z_{k+1}^i + \langle \log p_\theta^i(x_{k+1}^i), \pi\rangle  + \sum\limits_{j=1}^{n}a_{ij} \mathbb{E}_\pi \log \left(  \frac{1}{d{\mu}_{k}^j} \prod\limits_{l=1}^{n} \left(d{\mu}_{k}^l\right)^{a_{il}} \right) \nonumber  \\
		&
	=  -  \log Z_{k+1}^i + D_{KL}\left(\pi\|d{\mu}_{k+1}^i\right) - \sum\limits_{j=1}^{n}a_{ij} \mathbb{E}_\pi \log  d{\mu}_{k}^j  + \sum\limits_{l=1}^{n}a_{il} \mathbb{E}_\pi \log  d{\mu}_{k}^l  \nonumber  \\
	&= -  \log Z_{k+1}^i + D_{KL}\left(\pi\|d{\mu}_{k+1}^i\right) 
	\end{align}
	where $Z_{k+1}^i = \int_{\theta}  p_\theta^i(x_{k+1}^i) \prod_{j=1}^{n} (d\mu_k^j(\theta))^{a_{ij}}$ is the corresponding normalizing constant. 
	
	The first term in Eq.~\eqref{prop1} does not depend on the distribution $\pi$. Thus, 
	we conclude that the solution to the problem in Eq.~\eqref{distributed} is 
	the density $\pi^* = d{\mu}_{k+1}^i$ as defined in Eq.~\eqref{protocol} (almost everywhere).
	\qed
\end{proof}

We remark that the update in Eq.~\eqref{protocol} can be viewed as two-step processes: first  every agent constructs an aggregate belief using a weighted geometric average of its own belief and the beliefs of its neighbors, and then each agent performs a Bayes' update using the aggregated belief as a prior. We note that similar arguments in the context of distributed optimization have been proposed  in \cite{rab15,li16} for general Bregman distances. In the case when the number of hypotheses is finite, 
 variations on this update rule were previously analyzed in  \cite{sha14,ned15,lal14b}. 

\vspace{-0.4cm}
\subsection{An example} 

\begin{example}\label{ex1}
	 Consider a group of $4$ agents, connected over a network as shown in Figure \ref{network}. 
A set of metropolis weights for this network is given by the following matrix:
\begin{align*}
A & = \left[\begin{array}{c c c c}
2/3 & 1/6 &  0 & 1/6 \\
1/6 & 2/3 &  1/6 & 0 \\
0 & 1/6 &  2/3 & 1/6 \\
1/6 & 0 &  1/6 & 2/3
\end{array} \right]. 
\end{align*}
\begin{figure}[h!]
	\centering
	\begin{overpic}[width=0.4\textwidth]{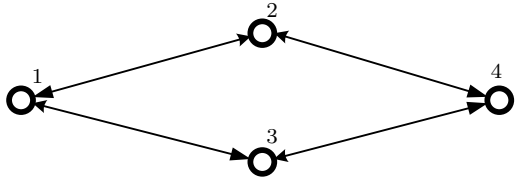}
		\put(4,19){{ $1$}}
		\put(50,32){{ $2$}}
		\put(50,7){{ $3$}}
		\put(94,20){{ $4$}}
	\end{overpic}
	\caption{A network of $4$ agents.}
	\label{network}
\end{figure}
Furthermore, assume that each agent is observing a Bernoulli random variable such that $X_k^1 \sim \text{Bern}(0.2)$, $X_k^2 \sim \text{Bern}(0.4)$, $X_k^3 \sim \text{Bern}(0.6)$ and $X_k^4 \sim \text{Bern}(0.8)$.
In this case, the parameter space is $\Theta =[0,1]$. Thus, the objective is to collectively find a parameter $\theta^*$ that best explains the joint observations in the sense of the problem in Eq. \eqref{opt_problem}, i.e.
\begin{align*}
\min_{\theta \in [0,1]} F(\theta) & = \sum_{j=1}^{4} D_{KL}(\text{Bern}(\theta)\|\text{Bern}(\theta^j))  = \sum_{j=1}^{4} \left( \theta \log\frac{\theta}{\theta^j} + (1-\theta)\log\frac{1-\theta}{1-\theta^j}\right) 
\end{align*}
where $\theta^1 = 0.2$, $\theta^2 = 0.4$, $\theta^3 = 0.6$ and $\theta^4 = 0.8$.
We  can be see that the optimal solution is $\theta^* = 0.5$ by determining it explicitly via the first-order optimality conditions or by exploiting the symmetry in the objective function.

Assume that all agents start with a common belief at time $0$ following a Beta distribution, i.e., 
$\mu_0^i = \text{Beta}(\alpha_0,\beta_0)$ (this specific choice will be motivated in the next section). 
Then, the proposed algorithm in Eq. \eqref{protocol} will generate a belief at time $k+1$ that also has a Beta distribution. Moreover, $\mu^i_{k+1} = \text{Beta}(\alpha_{k+1}^i,\beta_{k+1}^i)$, where
\begin{align*}
\alpha_{k+1}^i & = \sum_{j=1}^{n}a_{ij}\alpha_{k}^j + x_{k+1}^i, \qquad
\beta_{k+1}^i  =  \sum_{j=1}^{n}a_{ij}\beta_{k}^j + 1 - x^i_{k+1}.
\end{align*}
\end{example}

To summarize, we have given an interpretation of Bayes' rule as an instance of Stochastic Mirror Descent. We have shown how this interpretation motivates a distributed update rule.  In the next section, we discuss explicit forms of this update rule for parametric models coming from exponential families.

\section{Cooperative Inference for Exponential Families}\label{sec:inference}

We begin with the observation that, for a general class of models $\{\mathscr{P}^i\}$, it is not clear whether the computation of the posterior beliefs $\mu_{k+1}^i$ is tractable. Indeed, computation of $\mu_{k+1}^i$ involves solving an integral of the form
\begin{align}\label{integral_theta}
\int_{\Theta}p_\theta^i(x_{k+1}^i) \prod_{j=1}^{n} (d\mu_k^j(\theta))^{a_{ij}}.
\end{align}
There is an entire area of research called \textit{variational Bayes' approximations} dedicated to efficiently approximating integrals that appear in such context \cite{fox12,bea03,dai16}. 

The purpose of this section is to show that for exponential family~\cite{koo36,dar35} there are closed-form expressions for the posteriors. 

\begin{definition}
	The exponential family, for a parameter $\theta = [\theta^1,\theta^2,\hdots,\theta^s]'$, is the set of probability distributions whose density can be represented as 
	\begin{align*}
	p_\theta(x) & = H(x) \exp(M(\theta)'T(x)-C(\theta))
	\end{align*}
	for specific functions $H(\cdot)$, $M(\cdot)$, $T(\cdot)$ and $C(\cdot)$, with ${{M(\theta) = [M(\theta^1),M(\theta^2),\hdots,M(\theta^s)]'}}$. The function $M(\theta)$ is usually referred to as 
	the natural parameter.
	
	When $M(\theta)$ is used as a parameter itself, it is said that the distribution is in its canonical form. In this case, we can write the density as
	\begin{align*}
	p_M(x) & = H(x) \exp(M'T(x)-C(M)),
	\end{align*}
	with $M$ being the parameter.
\end{definition}

Among the members of the exponential family, one can find the distributions such as Normal, Poisson, Exponential, Gamma, Bernoulli, and Beta, among others~\cite{gel14}. In our case, 
we will take advantage of the existence of \textit{conjugate priors} for all members of the exponential family. The definition of the conjugate prior is given below.

\begin{definition}
	Assume that the prior distribution $p$ on a parameter space $\Theta$ belongs to the exponential family. Then, the distribution $p$ is referred to as the \textit{conjugate prior} for a likelihood function $p_\theta(x)$ if the posterior distribution $p(\theta|x) \propto p_\theta(x) p(\theta)$ is in the same family as the prior.
\end{definition}

Thus, if the belief density at some time $k$ is a conjugate prior for our likelihood model, then our belief at time $k+1$ will be of the same class as our prior. For example, if a likelihood function follows a Gaussian form, then having a Gaussian prior will produce a Gaussian posterior.  This property simplifies the structure of the belief update procedure, since we can express the evolution of the beliefs generated by the proposed algorithm in 
Eq.~\eqref{protocol} by the evolution of the natural parameters of the member of the exponential family it belongs to.

We now proceed to provide more details. 
First, the conjugate prior for a member of the exponential family can be written as
\begin{align*}
p_{\chi,\nu}(M) & = f(\chi,\nu)\exp(M'\chi - \nu C(M)),
\end{align*}
which is a distribution over the natural parameters $M$, where $\nu>0$ and $\chi \in \mathbb{R}^s$ are the parameters of the conjugate prior. Then, it can be shown that the posterior distribution,
given some observation $x$, has the same exponential form as the prior with updated parameters as
follows:
\begin{align}\label{posterior_expo}
p_{\chi,\nu}(M|x) & = p_{\chi+T(x),\nu + 1}(M) \propto p_{\theta}(x)p_{\chi,\nu}(M|x).
\end{align}

On the other hand, for a set on $n$ priors of the same exponential family, the weighted geometric averages 
also have a closed form in terms of the conjugate parameters.

\begin{proposition}\label{geo_expo}
	Let $(p_{\chi^1,\nu^1}(M),\hdots,p_{\chi^n,\nu^n}(M) )$ be a set of $n$ distributions, all in the same class in the exponential family, i.e.,  $p_{\chi^i,\nu^i}(M)=f(\chi^i,\nu^i)\exp(M'\chi^i - \nu^i C(M))$ for $i=1,\hdots,n$. Then, for a set $(\alpha_1,\hdots,\alpha_n)$ of weights with $\alpha_i >0$ for all $i$,
	the probability distribution defined as
	\begin{align*}
	p_{\bar \chi,\bar \nu}(M) & =\frac{\prod_{i=1}^{n}(p_{\chi^i,\nu^i}(M))^{\alpha_i}}{\int\prod_{j=1}^{n}(p_{\chi^j,\nu^j}(dM))^{\alpha_j}}, 
	\end{align*} 
	belongs to the same class in the exponential family with parameters $\bar \chi = \sum_{i=1}^{n}\alpha_i \chi^i$ and $\bar \nu = \sum_{i=1}^{n}\alpha_i \nu^i$.
\end{proposition}
\begin{proof}
	We write the explicit geometric product, and discard the constant terms
	\begin{align*}
	p_{\bar \chi,\bar \nu}(M) 
	& \propto \prod_{i=1}^{n}(f(\chi^i,\nu^i)\exp(M'\chi^i - \nu^i C(M)))^{\alpha_i} \\
	&  \propto \exp\left(M' \sum_{i=1}^{n}\alpha_i \chi^i - \sum_{i=1}^{n}\alpha_i\nu^i C(M)\right) .
	\end{align*}
	
	The last line provides explicit values for the parameters of the new distribution.
	\qed
\end{proof}	

%

The relations in Eq.~\eqref{posterior_expo} and Proposition~\ref{geo_expo} allow us to write the algorithm 
in~Eq.~\eqref{protocol} in terms of the natural parameters of the priors, as shown by the following proposition. 

\begin{proposition}\label{prop_conjugate}
	Assume that the belief density $d\mu_k^i$ at time $k$ has an exponential form with natural parameters 
	$\chi^i_k$ and $\nu^i_k$ for all $1 \leq i \leq n$, and that these densities are conjugate priors of the likelihood models $p^i_\theta$. Then, the belief density at time $k+1$, as computed in the update rule in Eq.~\eqref{protocol}, has the same form as the beliefs at time $k$ with the natural parameters given by
	\begin{equation*}\label{protocol_natural}
		\chi^i_{k+1} = \sum_{j=1}^{n}a_{ij} \chi^j_k + T^i(x^i), \quad
		\nu_{k+1}^i= \sum_{j=1}^{n}a_{ij} \nu^j_k + 1\qquad\hbox{for all } i=1,\ldots,n.
	\end{equation*}
\end{proposition}

The proof of Proposition \ref{prop_conjugate} follows immediately from Eq.~\eqref{posterior_expo} and Eq.~\eqref{geo_expo}.

Proposition \ref{prop_conjugate} simplifies the algorithm in Eq.~\eqref{protocol} and facilitates its use in traditional estimation problems where members of the exponential family are used. We next illustrate this by discussing a number of distributed estimation problems with likelihood models coming from exponential families. 

\vspace{-0.4cm}
\subsection{Distributed Poisson Filter}
Consider an observation model where the agent signals follow Poisson distributions, i.e., $X^i_k = \text{Poisson}(\lambda^i)$ for all $i$. In this case, the optimization problem to be solved is 
\begin{align*}
\min_{\lambda >0} F(\lambda) &  = \sum_{j=1}^{n} D_{KL}(\text{Poisson}(\lambda)\|\text{Poisson}(\lambda^j)) ,
\end{align*}
or equivalently,
$
\min_{\lambda >0} \{-\sum\limits_{i=1}^{n} \lambda^i\log \lambda + \lambda\}.
$

The conjugate prior of a Poisson likelihood model is the Gamma distribution. 
Thus, if at time $k$ the beliefs are given by $\mu_k^i  = \text{Gamma}(\alpha_k^i,\beta_k^i)$ for all $i$, 
then the beliefs at time $k+1$ are $\mu_{k+1}^i= \text{Gamma}(\alpha_{k+1}^i,\beta_{k+1}^i)$, where
\begin{align*}
\alpha_{k+1}^i & = \sum_{j=1}^{n}a_{ij}\alpha_{k}^j + x_{k+1}^i \qquad \text{and} \qquad
\beta_{k+1}^i  =  \sum_{j=1}^{n}a_{ij}\beta_{k}^j + 1 .
\end{align*}

\vspace{-0.4cm}
\subsection{Distributed Gaussian Filter with known variance}
Assume each agent observes a signal of the form $X^i_k = \theta^i +\epsilon^i_k$, where $\theta^i$ is finite and unknown, while $\epsilon^i \sim \mathcal{N}(0,1/\tau^i)$, with $\tau^i = 1/(\sigma^i)^2$, is known by agent $i$.
The optimization problem to be solved is 
\begin{align*}
\min_{\theta \in \mathbb{R}} F(\theta) &  = \sum_{j=1}^{n} D_{KL}(\mathcal{N}(\theta,1/\tau^j)\|\mathcal{N}(\theta^j,1/\tau^j)) ,
\end{align*}
or equivalently
$
\min_{\theta \in \mathbb{R}} \sum_{j=1}^{n} \tau^j(\theta - \theta^j)^2.
$

In this case, the likelihood models, the prior and the posterior are Gaussian. Thus, if the beliefs of the agents at time $k$ are Gaussian, i.e., $\mu_k^i = \mathcal{N}(\theta^i_k,1/\tau^i_k)$ for all $i=1\hdots,n$, 
then their beliefs at time $k+1$ are also Gaussian. In particular, they are given 
by ${{\mu_{k+1}^i = \mathcal{N}(\theta^i_{k+1},1/\tau^i_{k+1})}}$ for all $i=1\hdots,n$, with
\begin{align*}
\tau_{k+1}^i & =  \sum\limits_{j=1}^{n}a_{ij}\tau_k^j + \tau^i \qquad \text{and} \qquad
\theta_{k+1}^i = \frac{1}{\tau_{k+1}^i}\left( \sum\limits_{j=1}^{n} a_{ij} \tau_k^j  \theta_k^j + x_{k+1}^i\tau^i \right).
\end{align*}

We note that this specific setup is known a Gaussian Learning and has been studied in \cite{ned16,chu16}, where the expected parameter estimator is shown to converge at an $O(1/k)$ rate.

\vspace{-0.4cm}
\subsection{Distributed Gaussian Filter with unknown variance}

In this case, the agents want to cooperatively estimate the value of a variance. Specifically, based on observations of the form $X^i_k = \theta^i +\epsilon^i_k$, with $\epsilon^i_k \sim \mathcal{N}(0,1/\tau^i)$, where $\theta^i$ is known and $\tau^i$ is unknown to agent $i$, they want to solve the following problem
\begin{align*}
\min_{\tau > 0} F(\tau) &  = \sum_{j=1}^{n} D_{KL}(\mathcal{N}(\theta^j,1/\tau)\|\mathcal{N}(\theta^j,1/\tau^j)) .
\end{align*}

We choose the Scaled Inverse Chi-Squared\footnote{The density function of the Scaled Inverse Chi-Squared is defined for $x>0$ as ${{p_{\nu,\tau}(x) = \frac{(\tau v /2)^{v/2}}{\Gamma(v/2)}\frac{\exp(-\frac{-\nu \tau}{2x})}{x^{1+v/2}}}}$.} as the distribution of our prior, so that $\mu_k^i = \text{Scaled Inv}\text{-}\chi^2(\nu_k^i,\tau_k^i)$
for all $i$, then the beliefs at time $k+1$ are given by $\mu_{k+1}^i = \text{Scaled Inv}\text{-}\chi^2(\nu_{k+1}^i,\tau_{k+1}^i)$ for all $i$, with
\begin{align*}
\nu_{k+1}^i & = \sum\limits_{j=1}^{n}a_{ij}\nu_k^j+1 \qquad \text{and} \qquad
\tau_{k+1}^i  = \frac{1}{\nu_{k+1}^i}\left( \sum\limits_{j=1}^{n}a_{ij}\nu_k^j \tau_k^j +(x_{k+1}^i - \theta^i)^2\right) .
\end{align*}

\vspace{-0.4cm}
\subsection{Distributed Gaussian Filter with unknown mean and variance}

In the preceding examples, we have considered the cases when either the mean or the variance is known. Here, we will assume that both the mean and the variance are unknown and need to be estimated. 
Explicitly, we still have noise observations ${{X^i_k = \theta^i +\epsilon^i_k}}$, with $\epsilon^i_k \sim \mathcal{N}(0,1/\tau^i)$, and want to solve
\begin{align*}
\min_{\theta \in \mathbb{R},\tau > 0} F(\theta,\tau) &  = \sum_{j=1}^{n} D_{KL}(\mathcal{N}(\theta,1/\tau)\|\mathcal{N}(\theta^j,1/\tau^j)) .
\end{align*}

The Normal-Inverse-Gamma distribution serves as conjugate prior for the likelihood model over the parameters $(\theta,\tau)$. Specifically, we assume that the beliefs at time $k$ are given by
\begin{align*}
\mu_k^i &= \text{Normal-Inv-Gamma}(\theta^i_k,\tau_{k}^i,\alpha_k^i,\beta_k^i)\qquad\hbox{for all } i=1,\ldots,n.
\end{align*}
Then, the beliefs at time $k+1$ will have a Normal-Inverse-Gamma distribution with the following parameters
\begin{align*}
\tau_{k+1}^i & = \sum_{j=1}^{n}a_{ij}\tau_k^j +1, \qquad
\theta_{k+1}^i  = \frac{\sum_{j=1}^{n}a_{ij}\tau_k^j\theta_k^j +x^i_{k+1}}{\tau_{k+1}^i},\\
\alpha^i_{k+1} & = \sum_{j=1}^{n}a_{ij}\alpha_k^j + 1/2 ,\qquad  
\beta^i_{k+1}  = \sum_{j=1}^{n}a_{ij}\beta_{k}^j + \frac{\sum_{j=1}^{n}a_{ij}\tau^j_k(x^i_{k+1} - \theta^j_k)^2}{2\tau_{k+1}^i}.
\end{align*}

\section{Belief Concentration Rates}\label{sec:concentration}

We now turn to the presentation of our main results which concern the rate at which  beliefs generated by the update rule in Eq.~\eqref{protocol} concentrate around the true parameter $\theta^*$. We will break up our analysis into two cases. Initially, we will focus on the case when  $\Theta$  is a countable set, and will prove a concentration result for a ball containing the optimal hypothesis having finitely many hypotheses outside it. We will use this case to gently introduce the techniques we will use. We will then turn to our main scenario of interest,  namely when $\Theta$ is a compact subset of $\mathbb{R}^d$. Our proof techniques use concentration arguments for beliefs on Hellinger balls from the recent work~\cite{bir15} which, 
in turn, builds on the classic paper~\cite{lecam73}.

We begin with two subsections focusing on background information, definitions, and assumptions. 

\vspace{-0.4cm}
\subsection{Background: Hellinger Distance and Coverings}

We equip the set of all probability distributions over the parameter set $\mathscr{P}$ with the Hellinger distance\footnote{The Hellinger distance between two probability distributions $P$ and $Q$ is given by,
		\begin{align*}
		h^2\left(P,Q\right) & = \frac{1}{2} \int \left(\sqrt{\frac{dP}{d\lambda}}-\sqrt{\frac{dQ}{d\lambda}}\right)^2d\lambda,
		\end{align*}
		where $P$ and $Q$ are dominated by $\lambda$. Note that this formula is for the square of the Hellinger distance.}
	to obtain the {\it metric} space $\left(\mathscr{P},h\right)$.  The metric space induces a topology, where we can define an open ball $\mathcal{B}_r(\theta)$ with a radius $r>0$ centered at a point $\theta \in \Theta$, which we use to construct a special covering of subsets $B\subset \mathscr{P}$.
	\begin{definition}\label{h_balls}
		Define an $n$-Hellinger ball of radius $r$ centered at $\theta$ as
		\begin{align*}
		\mathcal{B}_r(\theta) & = \left\lbrace \hat\theta \in \Theta \left|   \sqrt{\frac{1}{n} \sum_{i=1}^{n} h^2\left(P^i_{\theta},P^i_{\hat\theta}\right) } \right. \leq r  \right\rbrace .
		\end{align*}
		Additionally, when no center is specified, it should be assumed that it refers to $\theta^*$, i.e. $\mathcal{B}_r = \mathcal{B}_r(\theta^*) $. 
\end{definition}

	Given an $n$-Hellinger ball of radius $r$, we will use the following notation for a covering of its complement $\mathcal{B}_{r}^c$. Specifically, we are going to express $\mathcal{B}_{r}^c$ as the union of finite disjoint and concentric anuli. 
	Let $r>0$ and $\{r_l\}$ be a finite strictly decreasing sequence such that ${r_1 =1}$ and $r_L = r$. Now, express the set $\mathcal{B}_{r}^c $ as the union of anuli generated by the sequence $\{r_l\}$ as 
	\begin{align*}
	\mathcal{B}_{r}^c & = \bigcup_{l = 1}^{L-1} \mathcal{F}_l,
	\end{align*}
	where $\mathcal{F}_l = \mathcal{B}_{r_{l}} \setminus\mathcal{B}_{r_{l+1}} $.


\vspace{-0.4cm}
\subsection{Background: Assumptions on Network and Mixing Weights}

Naturally, we need some assumptions on the matrix $A$. For one thing, the matrix $A$ has to be ``compatible'' with the underlying graph, in that information from node $i$ should not affect node $j$ if there is no edge from $i$ to $j$ in $\mathcal{G}$. At the other extreme, we want to rule out the possibility that $A$ is the identity matrix, which in terms of Eq. (\ref{protocol}) means nodes do not talk to their neighbors. Formally, we make the following assumption.

\begin{assumption}\label{assum:graph}
	The graph $\mathcal{G}$ and matrix $A$ are such that:
	\begin{enumerate}[(a)]
		\item $A$ is doubly-stochastic with $\left[A\right]_{ij} = a_{ij} > 0$ for  $i\ne j$ if and only if $(i,j)\in E$.
		\item $A$ has positive diagonal entries, $a_{ii}>0$ for all $i \in V $.
		\item The graph $\mathcal{G}$ is  connected.
	\end{enumerate}
\end{assumption} 

Assumption~\ref{assum:graph} is common in the distributed optimization literature. The construction of a set of weights satisfying Assumption~\ref{assum:graph} can be done in a distributed way, for example, 
by choosing the so-called ``lazy Metropolis'' matrix, which is a stochastic matrix given by
\begin{align*}
a_{ij} = \left\{  \begin{array}{l l}
\frac{1}{2 \max \left\{ d^i+1,d^j+1 \right\}  } & \quad \text{if $(i,j) \in E$},\\
0 & \quad \text{if $(i,j) \notin E$},
\end{array} \right. 
\end{align*}
where $d^i$ is the degree (the number of neighbors) of node $i$. Note that although the above formula only gives the off-diagonal entries of $A$, it uniquely defines the entire matrix (the diagonal elements are uniquely defined via the stochasticity of $A$). To choose the weights corresponding to a lazy Metropolis matrix, agents will need to spend an additional round  at the beginning of the algorithm broadcasting their degrees to their neighbors.

Assumption 1 can be seen to guarantee that $A^t \rightarrow (1/n) {\bf 1} {\bf 1}^T$ where ${\bf 1}$ is the vector of all ones. We will use the following result that provides convergence rate for the difference 
$|A^t - (1/n) {\bf 1} {\bf 1}^T|$, based on the results from~\cite{sha14} and~\cite{ned15}:
	\begin{lemma}\label{shahin}
		Let Assumption \ref{assum:graph} hold, then the matrix $A$ satisfies the following relation:
		\begin{align*}
		\sum\limits_{t=1}^{k} \sum\limits_{j=1}^{n} \left| \left[A^{k-t}\right]_{ij} - \frac{1}{n} \right| & \leq \frac{4 \log n}{1-\delta}  \qquad \mbox{ for } i = 1, \ldots, n,
		\end{align*}
		where $\delta = 1-\eta / 4n^2$ with $\eta$ being the smallest positive entry of the matrix $A$. Furthermore, if $A$ is a lazy Metropolis matrix associated with the graph $\mathcal{G}$, 
		then $\delta = 1- 1 / \mathcal{O}(n^2)$.
	\end{lemma}
	
\vspace{-0.4cm}
\subsection{Concentration for the Case of Countable Hypotheses}

We now turn to proving a concentration result when the set $\Theta$ of hypotheses is countable. We will consider the case of a ball in the Hellinger distance containing a countable number of hypotheses, including the correct one, and having only finitely many hypotheses outside it; we will show exponential convergence of beliefs to that ball. The purpose is to gently introduce the techniques we will use later in the case of a compact set of hypotheses. 

 In the case when the number of hypotheses is countable,  the density update in Eq.~\eqref{protocol} can be restated in a simpler form for discrete beliefs over the parameter space $\Theta$ as
\begin{align}\label{protocol_dis}
\mu_{k+1}^i(\theta) & \propto  p_\theta^i(x_{k+1}^i) \prod_{j=1}^{n} (\mu_k^j(\theta))^{a_{ij}} .
\end{align}

We will fix the radius $r$, and our goal will be to prove a concentration result for a Hellinger ball of radius $r$ around the optimal hypothesis $\theta^*$. We partition the complement of this ball $\mathcal{B}_r^c$ as described above into annuli $\mathcal{F}_l$. We introduce the notation $\mathcal{N}_l$ to denote the number of hypotheses within the annulus $\mathcal{F}_l$. We refer the reader to Figure \ref{lcovering_dis} which shows a set of probability distributions, represented as black dots, where the true distribution $P$ is represented by a star. 

\begin{figure}[ht]
	\centering
	\begin{overpic}[width=0.4\textwidth]{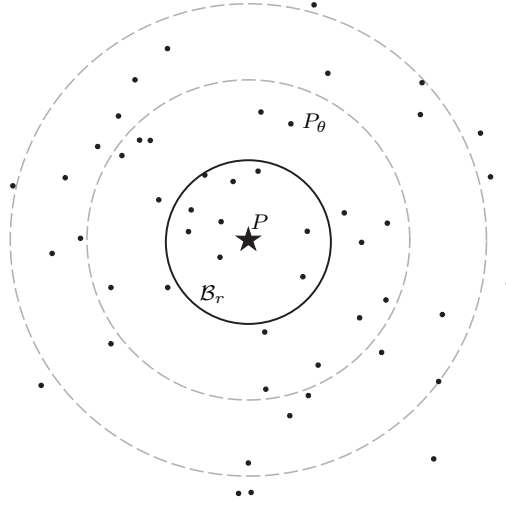}
		\put(38,39){{\small $\mathcal{B}_r $}}
		\put(48,53){{\small $P$}}
		\put(58,73){{\small $P_{\theta}$}}
	\end{overpic} 
	\caption{Creating a covering for a ball $\mathcal{B}_r $. $\bigstar$ represents the correct hypothesis 
		$\boldsymbol{P}_{\theta^*}$, $\bullet$ indicates the location of other hypotheses and the dash lines indicate the boundary of the balls $\mathcal{B}_{r_l} $.}
	\label{lcovering_dis}
\end{figure}

We will assume that the number of hypotheses outside the desired ball is finite.

\begin{assumption}\label{assum:conv_hyp} 	
The number of hypothesis outside $\mathcal{B}_r$ is finite. 
\end{assumption}

Additionally, we impose a bound on the separation between hypotheses which will avoid some pathological cases. The separation between hypotheses is defined in terms of the Hellinger affinity between two distributions
$Q$ and $P$, given by 
\[\rho(Q,P) = 1 - h^2(Q,P).\]

\begin{assumption}\label{assum:bound}
	There exists an $\alpha > 0 $ such that 
	$\rho(P_{\theta_1}^i,P_{\theta_2}^i)>\alpha$ for any $\theta_1, \theta_2 \in \Theta$ and $i = 1,\hdots,n$ .
\end{assumption}

With these assumptions in place, our first step is a lemma that bounds concentration of log-likelihood ratios.

\begin{lemma}\label{lemmab}
	Let Assumptions \ref{assum:graph}, \ref{assum:conv_hyp} and \ref{assum:bound} hold. Given a set of independent random variables $\{X^i_t\}$ such that $X^i_t \sim P^i$ for $i=1,\hdots,n$ and $t =1,\hdots, k$, a set of distributions $\{Q^i\}$ where $Q^i$ dominates $P^i$, then for all $y\in \mathbb{R}$,
	\begin{align*}
	\mathbb{P}\left[\sum\limits_{t=1}^{k} \sum\limits_{j=1}^{n} [A^{k-t}]_{ij} \log \frac{dQ^j}{dP^j}(X^j_t) \geq y \right]  
	& \leq \exp(-y / 2)  \exp\left(\log \frac{1}{\alpha} \frac{4 \log n}{1- \delta} \right)   \exp\left( -k \frac{1}{n}\sum\limits_{j=1}^{n} h^2(Q^j,P^j)\right).
	\end{align*}
\end{lemma}
\begin{proof}
	By the Markov inequality and Jensen's inequality we have 
	\begin{align*}
	\mathbb{P}\left[\sum\limits_{t=1}^{k} \sum\limits_{j=1}^{n} [A^{k-t}]_{ij} \log \frac{dQ^j}{dP^j}(X^j_t) \geq y \right]  & \leq \exp(-y / 2)\mathbb{E}\left[\prod\limits_{t=1}^{k} \prod\limits_{j=1}^{n} \sqrt{ \left( \frac{dQ^j}{dP^j}(X^j_t)\right) ^{[A^{k-t}]_{ij}}}\right] \\
	& \leq \exp(-y / 2) \prod\limits_{t=1}^{k} \prod\limits_{j=1}^{n} \mathbb{E}\left[ \sqrt{ \left( \frac{dQ^j}{dP^j}(X^j_t)\right) }\right]^{[A^{k-t}]_{ij}} \\
	& \leq \exp(-y / 2) \prod\limits_{t=1}^{k}\prod\limits_{j=1}^{n}\rho(Q^j,P^j)^{[A^{k-t}]_{ij}},
	\end{align*}
	where the last inequality follows from the definition of the Hellinger affinity function $\rho(Q,P)$. 
	Now, by adding and subtracting $\frac{1}{n}\sum_{j=1}^{n}\log \rho(Q^j,P^j)$ we have
	\begin{align*}
	\mathbb{P}\left[\sum\limits_{t=1}^{k} \sum\limits_{j=1}^{n} [A^{k-t}]_{ij} \log \frac{dQ^j}{dP^j}(X^j_t) \geq y \right]  
	& \leq \exp(-y / 2) \exp \left( \sum\limits_{t=1}^{k} \sum\limits_{j=1}^{n}([A^{k-t}]_{ij}-1/n)\log \rho(Q^j,P^j)^{} + \sum\limits_{t=1}^{k} \frac{1}{n} \sum\limits_{j=1}^{n}\log \rho(Q^j,P^j)^{} \right)  \\
	& \leq \exp(-y / 2) \exp \left(\log \frac{1}{\alpha}\sum\limits_{t=1}^{k} \sum\limits_{j=1}^{n}|[A^{k-t}]_{ij}-1/n| + \sum\limits_{t=1}^{k} \frac{1}{n} \sum\limits_{j=1}^{n}\log \rho(Q^j,P^j)^{} \right), 
	\end{align*}
	where the last line follows from $\rho(P^j,Q^j)>\alpha$.
	
	Then, from Lemma \ref{shahin} it follows that
	\begin{align*}
	\mathbb{P}\left[\sum\limits_{t=1}^{k} \sum\limits_{j=1}^{n} [A^{k-t}]_{ij} \log \frac{dQ^j}{dP^j}(X^j_t) \geq y \right]  & \leq \exp(-y / 2) \exp \left( \log \frac{1}{\alpha}\frac{4 \log n}{1- \delta}  + \sum\limits_{t=1}^{k} \frac{1}{n} \sum\limits_{j=1}^{n}\log \rho(Q^j,P^j)^{} \right)  \\
	& \leq \exp(-y / 2)  \exp\left( \log \frac{1}{\alpha}\frac{4 \log n}{1- \delta} \right)   \prod\limits_{j=1}^{n}\rho^k(Q^j,P^j)^{1 /n}   \\
	& \leq \exp(-y / 2)  \exp\left(\log \frac{1}{\alpha} \frac{4 \log n}{1- \delta} \right)   \prod\limits_{j=1}^{n}\exp(-k h^2(Q^j,P^j))^{1 /n}.
	\end{align*}
	The last inequality follows from $\rho(Q^j,P^j) = 1 - h^2(Q^j,P^j)$ and $1-x \leq \exp(-x)$ for $x \in [0,1]$. 
	\qed
\end{proof}

We are now ready to state our first main result, which bounds concentration of Eq. ~\eqref{protocol_dis} around the optimal hypothesis for a countable hypothesis set $\Theta$. The following theorem shows that the beliefs of all agents will concentrate around the Hellinger ball $\mathcal{B}_r$ at an exponential rate.

\begin{theorem}\label{main_count}
	Let Assumptions~\ref{assum:graph}, \ref{assum:conv_hyp} and \ref{assum:bound} hold, 
	and let $\sigma \in (0,1)$ be a desired probability  tolerance. 
	Then, the belief sequences $\{\mu_{k}^i\}$, $i\in V$ that are generated by the update rule in Eq.~\eqref{protocol_dis},
	with initial beliefs such that ${\mu_0^i(\theta^*) >\epsilon}$ for all $i$,  
	have the following property: for any radius $r>0$ with probability $1-\sigma$,
	\begin{align*}
	\mu_{k+1}^i\left(\mathcal{B}_r \right) & \geq 1 - \frac{1}{\epsilon}\chi  \exp \left(  -k r^2 \right) \qquad \text{ for all } i \text{ and all }k\geq N,  
	\end{align*}
		where  
	\begin{align*}
	N=& \inf \left\lbrace t \geq 1 \Bigg| \exp\left( \log\frac{1}{\alpha}\frac{4 \log n}{1- \delta} \right) \sum\limits_{l = 1}^{L-1}\mathcal{N}_{r_l} \exp \left(  -tr_{l+1}^2 \right) < \rho\right\rbrace, 
	\end{align*}
	$\chi = \Sigma_{l =1}^{L-1}  \exp\left(-\frac{1}{2}r_{l}^2 +\log \mathcal{N}_{r_l}\right)$, 
	${\delta = 1-\eta/n^2}$, and $\eta$ is the smallest positive element of the matrix $A$.
\end{theorem}
\begin{proof}
	We are going to focus on bounding the beliefs of a measurable set $B$, such that $\theta^* \in B$. For such a set, it follows from Eq.~\eqref{protocol_dis} that
	\begin{align*}
	\mu_{k}^i\left(B\right) & = \frac{1}{Z_{k}^i}\sum\limits_{\theta \in B} 
	\left(\prod\limits_{j=1}^{n} \mu_0^j\left(\theta\right)^{\left[A^{k}\right]_{ij} } \right)
	\prod\limits_{t=1}^{k} \prod\limits_{j=1}^{n}  p^j_{\theta}(X_{t}^j)^{\left[A^{k-t}\right]_{ij}},
	\end{align*}
	where $ Z_{k}^i $ is the appropriate normalization constant. 
	
	Furthermore, after a few algebraic operations we obtain
	\begin{align*}
	\mu_{k}^i\left(B\right) & \geq 1 - \sum\limits_{\theta \in B^c} { \prod\limits_{j=1}^{n} \left( \frac{\mu_0^j(\theta)}{\mu_0^j(\theta^*) }\right) ^{\left[A^{k}\right]_{ij} }} \prod\limits_{t=1}^{k} \prod\limits_{j=1}^{n} \left(  \frac{p^j_{\theta}(X_{t}^j)}{p^j(X_{t}^j)}\right) ^{\left[A^{k\text{-}t}\right]_{ij}}.
	\end{align*}
	
	Moreover, since $\mu_0^i(\theta^*)  > \epsilon$ for all $i=1,\hdots,n$, it follows that
	\begin{align}\label{ready_for_bounds2}
	\mu_{k}^i\left(B\right) & \geq 1 - \frac{1}{\epsilon}\sum\limits_{\theta \in B^c}\prod\limits_{t=1}^{k} \prod\limits_{j=1}^{n} \left(  \frac{p^j_{\theta}(X_{t}^j)}{p^j(X_{t}^j)}\right) ^{\left[A^{k\text{-}t}\right]_{ij}}.
	\end{align}
	
	The relation in Eq.~\eqref{ready_for_bounds2} describes the iterative averaging of products of density functions, for which we can use Lemma \ref{lemmab} with $Q=P_\theta$ and $P = P_{\theta^*}$. Then,
	\begin{align*}
	\mathbb{P}\left( \left\lbrace \boldsymbol{X}^k \Bigg| \sup_{\theta \in B^c}\sum\limits_{t=1}^{k} \sum\limits_{j=1}^{n} [A^{k-t}]_{ij} \log \frac{p_{\theta}^j(X^j_t)}{p^j(X^j_t)} \geq y \right\rbrace \right)   &   \leq \sum_{\theta \in B^c} \exp(-y / 2)  \exp\left(\log\frac{1}{\alpha} \frac{4 \log n}{1- \delta} \right)   \exp\left( -k \frac{1}{n}\sum\limits_{j=1}^{n} h^2(P_{\theta}^j,P^j)\right)
	\end{align*}
	and by setting $y =-k \frac{1}{n}\sum\limits_{j=1}^{n} h^2(P_{\theta}^j,P^j)$ we obtain
	\begin{align*}
	&\mathbb{P}\left( \left\lbrace \boldsymbol{X}^k \Bigg| \sup_{\theta \in B^c}\sum\limits_{t=1}^{k} \sum\limits_{j=1}^{n} [A^{k-t}]_{ij} \log \frac{p_{\theta}^j(X^j_t)}{p^j(X^j_t)} \geq -k \frac{1}{n}\sum\limits_{j=1}^{n} h^2(P_{\theta}^j,P^j) \right\rbrace \right)  \\
	& \qquad \qquad \leq \exp\left(\log\frac{1}{\alpha} \frac{4 \log n}{1- \delta} \right) \sum_{\theta \in B^c} \exp\left( -\frac{k}{2} \frac{1}{n}\sum\limits_{j=1}^{n} h^2(P_{\theta}^j,P^j)\right).
	\end{align*}
	
	Now, we let the set $B$ be the Hellinger ball of a radius $r$ centered at $\theta^*$ 
	and define a cover (as described above) to exploit the representation of 
	$\mathcal{B}_r^c $ as the union of concentric Hellinger annuli, for which we have
	\begin{align*}
	&\mathbb{P}\left( \left\lbrace \boldsymbol{X}^k \Bigg| \sup_{\theta \in B^c}\sum\limits_{t=1}^{k} \sum\limits_{j=1}^{n} [A^{k-t}]_{ij} \log \frac{p_{\theta}^j(X^j_t)}{p^j(X^j_t)} \geq -k \frac{1}{n}\sum\limits_{j=1}^{n} h^2(P_{\theta}^j,P^j) \right\rbrace \right)\\
& \qquad \qquad \leq \exp\left( \log\frac{1}{\alpha}\frac{4 \log n}{1- \delta} \right) \sum\limits_{l = 1}^{L-1} \sum_{\theta \in \mathcal{F}_l} \exp\left( -\frac{k}{2} \frac{1}{n}\sum\limits_{j=1}^{n} h^2(P_{\theta}^j,P^j)\right)\\
& \qquad \qquad \leq \exp\left( \log\frac{1}{\alpha}\frac{4 \log n}{1- \delta} \right) \sum\limits_{l = 1}^{L-1}\mathcal{N}_{r_l} \exp \left(  -kr_{l+1}^2 \right).
	\end{align*} 

	We are interested in finding a value of $k$ large enough such that the above probability is below $\sigma$. Thus, lets define the value of $N$ as
	\begin{align*}
	N=& \inf \left\lbrace t \geq 1 \Bigg| \exp\left( \log\frac{1}{\alpha}\frac{4 \log n}{1- \delta} \right) \sum\limits_{l = 1}^{L-1}\mathcal{N}_{r_l} \exp \left(  -tr_{l+1}^2 \right) < \sigma\right\rbrace .
	\end{align*}
	
	It follows that for all $k\leq N$ with probability $1-\sigma$, for all $\theta \in \mathcal{B}_{r}^c $
	\begin{align*}
	\sum\limits_{t=1}^{k} \sum\limits_{j=1}^{n} [A^{k-t}]_{ij} \log \frac{p_{\theta}^j(X^j_t)}{p^j(X^j_t)} \leq -k \frac{1}{n}\sum\limits_{j=1}^{n} h^2(P_{\theta}^j,P^j).
	\end{align*}
	
	Thus, from Eq.~\eqref{ready_for_bounds2} with probability $1-\sigma$ we have
	\begin{align*}
	\mu_{k}^i\left(\mathcal{B}_{r} \right)
	& \geq 1 - \frac{1}{\epsilon}\sum_{\theta \in \mathcal{B}_{r}^c } \exp \left(  -k \frac{1}{n}\sum\limits_{j=1}^{n} h^2(P_{\theta}^j,P^j) \right)  \\
	&= 1 - \frac{1}{\epsilon}\sum\limits_{l =1}^{L-1}\sum_{\theta \in \mathcal{F}_l} \exp \left(  -k \frac{1}{n}\sum\limits_{j=1}^{n} h^2(P_{\theta}^j,P^j)  \right)\\
	&  \geq 1 -\frac{1}{\epsilon}\sum\limits_{l =1}^{L-1} \mathcal{N}_{r_{l}} \exp \left(  -kr_{l+1}^2 \right)\\
	&  \geq 1 -\frac{1}{\epsilon}\sum\limits_{l = 1}^{L-1} \mathcal{N}_{r_l}  \exp \left(  -r_{l+1}^2 \right)\exp \left(  -(k-1)r_{l}^2 \right)\\
	&  \geq 1 - \chi  \frac{1}{\epsilon}\exp \left(  -(k-1) r^2 \right),
	\end{align*}
	where $\chi = \Sigma_{l =1}^{L-1}  \exp\left(-\frac{1}{2}r_{l}^2 +\log \mathcal{N}_{r_l}\right)$.
	\qed
\end{proof}

\vspace{-0.4cm}
\subsection{A Concentration Result for a Compact Set of Hypotheses}

Next we consider the case when the hypothesis set $\Theta$ is a compact subset of $\mathbb{R}^d$. We will now additionally require the map from $\Theta$ to $\prod_{i=1}^n P_{\theta}^i$  be continuous (where the topology on the space of distributions comes from the Hellinger metric). 
This will be useful in defining coverings, which will be made clear  shortly.

\begin{definition}\label{separated}
	Let $\left(M,d\right)$ be a metric space. A subset $S \subseteq M$ is called $\varepsilon$-separated with $\varepsilon>0$ if $d(x,y)\geq \varepsilon$ for any $x,y \in S$. Moreover, for a set 
	$B \subseteq M$, let $N_B(\varepsilon)$ be the smallest number of Hellinger balls with centers in $S$ 
	of radius $\varepsilon$ needed to cover the set $B$, i.e., such that $B \subseteq \bigcup_{ m \in S} \mathcal{B}_{\varepsilon}\left(m\right)$.

As before, given a decreasing sequence $1=r_1 \geq r_2 \geq \cdots \geq r_L = r$, we will define the annulus $\mathcal{F}_l$ to be $\mathcal{F}_l = \mathcal{B}_{r_{l}} \setminus\mathcal{B}_{r_{l+1}}$. Furthermore, $S_{\varepsilon_l}$ will denote maximal $\varepsilon_l$-separated subset of $\mathcal{F}$. Finally, $K_l = |S_{\varepsilon_l}|$.  
\end{definition}
	
	We note that, as a consequence of our assumption that the map from $\Theta$ to $\prod_{i=1}^n P_{\theta}^i$ is continuous, we have that each $K_l$ is finite (since the image of a compact set under a continuous map is compact). 	Thus, we have the following covering of $\mathcal{B}_{r}^c $:
	\begin{align*}
	\mathcal{B}_{r}^c  & = \bigcup_{l = 1}^{L-1}\bigcup_{m \in S_{\varepsilon_l}} \mathcal{F}_{l,m} ,
	\end{align*}
	where each $\mathcal{F}_{l,m}$ is the intersection of a ball in $S_{\varepsilon_l}$ with $\mathcal{F}_l$. Figure \ref{lcovering_cont} shows the elements of a covering for a set $\mathcal{B}_r^c $. The cluster of circles at the top right corner represents the balls $ \mathcal{B}_{\varepsilon_l}$ and, for a specific case in the left of the image, we illustrate the set $\mathcal{F}_{l,m}$.

\begin{figure}[h]
	\centering
	\begin{overpic}[width=0.4\textwidth]{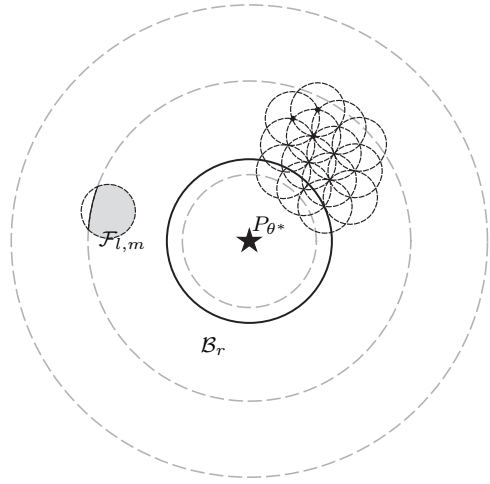}
		\put(38,29){{\small $\mathcal{B}_{r} $}}
		\put(18,50){{\small $\mathcal{F}_{l,m}$}}
		\put(48,53){{\small $P_{\theta^*}$}}
	\end{overpic} 
	\caption{Creating a covering for a set $\mathcal{B}_r $. $\bigstar$ represents the correct hypothesis $\boldsymbol{P}_{\theta^*}$. }
	\label{lcovering_cont}
\end{figure}

 \begin{example}
We continue Example \ref{ex1} from Section \ref{sec:variational}. Suppose we are interested in analyzing the concentration of the beliefs around the true parameter $\theta^*$ on a Euclidean ball of radius $0.05$; that is we want to see the total mass on the set $[0.45,0.55]$. This in turn, represents a Hellinger ball of radius $r=0.001254$. For this choice of $r$, we propose a covering where $r_1 = 1$, $r_2 = 1/2$, $r_3 = 1/4$, $\hdots$, $r_{10} =1/512$, $r_{11} =r$.
\begin{figure}[ht]
	\centering
	\begin{overpic}[width=0.7\textwidth]{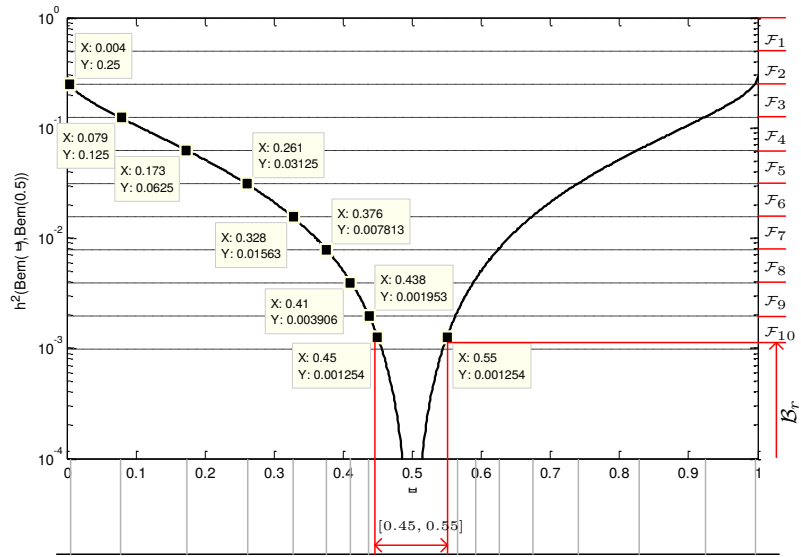}
		\put(47.7,3){{\tiny $[0.45,0.55]$}}
		\put(91,58){{\tiny $\mathcal{F}_1$}}
		\put(91,54.5){{\tiny $\mathcal{F}_2$}}
		\put(91,51){{\tiny $\mathcal{F}_3$}}
		\put(91,47.1){{\tiny $\mathcal{F}_4$}}
		\put(91,43.6){{\tiny $\mathcal{F}_5$}}
		\put(91,39.9){{\tiny $\mathcal{F}_6$}}
		\put(91,36.3){{\tiny $\mathcal{F}_7$}}
		\put(91,32.3){{\tiny $\mathcal{F}_8$}}
		\put(91,28.5){{\tiny $\mathcal{F}_9$}}
		\put(91,25){{\tiny $\mathcal{F}_{10}$}}
		\put(93,15){{\rotatebox{90}{$\mathcal{B}_r $}}}
	\end{overpic}
	\caption{Hellinger distance of the density $p_\theta$ to the optimal density $p_{\theta^*}$.}
	\label{hell_ball}
\end{figure}

Figure \ref{hell_ball} shows the Hellinger distance between the hypotheses $p_\theta$ and the optimal one $p_{\theta^*}$. Specifically, the $x$-axis is the value of $\theta$, and the $y$-axis shows the Hellinger distance between the distributions. Figure \ref{hell_ball} also shows the covering we defined before, as horizontal lines for each value of the sequence $r_l$, which in turn defines the annulus $\mathcal{F}_l$. T
he Hellinger ball of radius $r$ is also shown, with the corresponding subset of $\Theta$ where we want to analyze the belief concentration.

In this example, the parameter has dimension $1$. The number of balls needed to cover each annulus can be seen to be 2, i.e., we only need $2$ balls of radius $r_l/2$ to cover the annulus $\mathcal{F}_l$. 
Thus, $K_l=2$ for $1 \leq l \leq L-1$.
\qed
\end{example}

Our concentration result requires the following assumption on the densities.

\begin{assumption}\label{assum:max}
	For every $i = 1,\hdots,n$ and all $\theta$, it holds that $p_\theta^i(x)\leq 1$ almost everywhere.
\end{assumption}

Assumption \ref{assum:max} will be technically convenient for us. It can be made without loss of generality in the following sense: we can always modify the underlying problem to make it hold. 

Let us give an example before explaining the  reasoning behind this assertion. Let us assume there is just one  agent, and say $X \sim P$ is Gaussian with mean $\theta^* = 5$ and variance $0.01$. Our model is  $P_\theta = \mathcal{N}(\theta,0.01)$ for $\theta \in \Theta = [0,10]$. Because the variance is so small, the density values are larger than $1$. Instead let us multiply all our observations by $10$. We will then have that our observations come from $10X$, which indeed has density upper bounded by one. In turn our model now should be  $Q_\theta = \mathcal{N}(10\theta,1)$ or, alternatively, $Q_\theta = \mathcal{N}(\theta,1)$ for $\theta \in \hat \Theta = [0,100]$.

We note that this modification does not come without cost. As in the case of countable hypotheses, 
our convergence rates will depend on $\alpha$, defined to be a positive number such that 
$\rho(P_{\theta^1},P_{\theta^2})>\alpha$ for any $\theta^1$ and $\theta^2$. 
The process we have sketched out can decrease this parameter $\alpha$. 

In the general case, if each agent observe $X_t^j \sim P^j$, then there exists a large enough constant $M >1$ such that ${{MX_t^j \sim Q^j}}$ where the density of $Q^j$ is at most $1$. We can then have agents multiply their measurements by $M$ and redefine the densities to account for this. 

We next provide a concentration result for the logarithmic likelihood of a ratio of densities, which will serve the same technical function as Lemma \ref{lemmab} in the countable hypothesis case. 
We begin by defining two measures. For a hypothesis~$\theta$ and a measurable set $B \subseteq \Theta$, let ${\boldsymbol{P}}_{B}^{\otimes k}$ be the probability distribution with density
\begin{align}\label{density_g}
g_B(\boldsymbol{x}^k) & = \frac{1}{\mu_0(B)} \int\limits_{B} \prod\limits_{t=1}^{k} \prod\limits_{j=1}^{n} p_\theta^j(x_t^j)d\mu_{0}(\theta).
\end{align} 
Similarly, let $\bar{\boldsymbol{P}}_{B}^{\otimes k}$ be the measure with density (i.e., Radon-Nikodym derivative with respect to $\lambda^{\otimes nk}$),
	\begin{align}\label{bar_density_g}
\bar g_B(\boldsymbol{x}^k) & = \frac{1}{\mu_0(B)} \int\limits_{B} \prod\limits_{t=1}^{k} \prod\limits_{j=1}^{n} ( p_\theta^j(x_t^j))^{[A^{k-t}]_{ij}}d\mu_{0}(\theta).
\end{align} 

Note that $\bar{\boldsymbol{P}}_{B}^{\otimes k}$'s are not probability distributions due to the exponential weights. Nonetheless, they are bounded and positive. The next lemma shows the concentration of the logarithmic ratio of two weighted densities, as defined in Eq.~\eqref{bar_density_g}, for two different sets $B_1$ and $B_2$, in terms of the probability distribution ${\boldsymbol{P}}_{B_1}^{\otimes k}$.

\begin{lemma}\label{lemma2}
	Let Assumptions \ref{assum:graph}, \ref{assum:bound} and \ref{assum:max} hold. Consider two measurable sets $B_1,B_2 \subset \Theta$, both with positive measures, and assume that $B_{1} \subset \mathcal{B}_{r^{1}}(\theta^{1})$ and $B_{2} \subset \mathcal{B}_{r^{2}}(\theta^{2})$ {where $\mathcal{B}_{r^{1}}(\theta^{1})$ and $\mathcal{B}_{r^{2}}(\theta^{2})$ are disjoint}. Then, for all $y \in \mathbb{R}$
	\begin{align*}
	{\mathbb{P}}_{B_1}\left[\log \frac{\bar g_{B_2}(\boldsymbol{X}^k)}{\bar g_{B_1}(\boldsymbol{X}^k)} \geq y\right] 
	& \leq  \exp(- y / 2) \exp \left(\log \frac{1}{\alpha} \frac{4 \log n}{1- \delta} \right) \exp\left( -k \left( \sqrt{\frac{1}{n}\sum_{j=1}^{n} h^2(P^j_{\theta^1},P^j_{\theta^2})} -r^1 -r^2\right)^2\right),
	\end{align*} 
	where ${\mathbb{P}}_{B_1}$ is the probability measure that gives $\boldsymbol{X}^k$ having a distribution ${\boldsymbol{P}}_{B_1}^{\otimes k}$  with density $ g_{B_1}$ as defined in Eq.~\eqref{density_g}.
\end{lemma}
\begin{proof}
	By the Markov inequality, it follows that
	\begin{align*}
	{\mathbb{P}}_{B_1}\left[\log \frac{\bar g_{B_2}(\boldsymbol{X}^k)}{\bar g_{B_1}(\boldsymbol{X}^k)} \geq y\right] & \leq \exp(-y/2)\mathbb{E}_{B_1}\left[ \sqrt{\frac{\bar g_{B_2}(\boldsymbol{X}^k)}{\bar g_{B_1}(\boldsymbol{X}^k)}}\right] \nonumber\\
	&  = \exp(-y/2) \int_{\boldsymbol{\mathcal{X}}^k} \sqrt{\frac{\bar g_{B_2}(\boldsymbol{x}^k)}{\bar g_{B_1}(\boldsymbol{x}^k)}} g_{B_1}(\boldsymbol{x}^k) d\lambda^{\otimes kn}(\boldsymbol{x}^k).
	\end{align*}
	
	Now, by Assumption \ref{assum:max} it follows that $g_{B} \leq \bar g_{B}$ almost everywhere. Thus,  we have
	\begin{align*}
{\mathbb{P}}_{B_1}\left[\log \frac{\bar g_{B_2}(\boldsymbol{X}^k)}{\bar g_{B_1}(\boldsymbol{X}^k)} \geq y\right]&  \leq  \exp(-y/2) \int_{\boldsymbol{\mathcal{X}}^k} \sqrt{ \bar g_{B_2}(\boldsymbol{x}^k)} \sqrt{\bar g_{B_1}(\boldsymbol{x}^k)}  d\lambda^{\otimes kn}(\boldsymbol{x}^k) \\
& \leq \exp(-y/2) \rho\left( \bar {\boldsymbol{P}}_{B_2}^{\otimes k},\bar{\boldsymbol{P}}_{B_1}^{\otimes k} \right),
\end{align*}
where we are interpreting the definition of the Hellinger affinity function $\rho(\cdot,\cdot)$ as a function of two bounded positive measures, not necessarily probability measures.
	
At this point, we can follow the same argument as in Lemma~$2$ in \cite{lecam86}, page $477$, where the Hellinger affinity of two members of the convex hull of sets of probability distributions is shown to be less than the product of the Hellinger affinity of the factors. In our particular case, the measures $\bar{\boldsymbol{P}}_{B}^{\otimes k}$ are not probability distributions, nonetheless, the same disintegration argument holds. Thus, we obtain
	\begin{align*}
\rho\left( \bar {\boldsymbol{P}}_{B_2}^{\otimes k},\bar{\boldsymbol{P}}_{B_1}^{\otimes k} \right)
	& \leq \prod_{t=1}^k\prod_{j=1}^{n}\rho\left( \bar{P}_{B_2}^j,\bar{P}_{B_1}^j \right) ,
	\end{align*}
	where $\bar{P}_{B}^j$ is the measure with Radon-Nikodym derivative $\bar	g_{B}(x)  = \frac{1}{\mu_{0}(B)} \int\limits_{B} (p_\theta^j(x))^{[ A^{k-t}]_{ij}} d\mu_{0}(\theta)$ with respect to $\lambda$.
	
	In addition, by Jensen's inequality\footnote{For a concave function $\phi$ and $\int_{\Omega}f(x)dx =1$, it holds that $\int_{\Omega}\phi(g(x))f(x)dx \leq \phi\left( \int_{\Omega}g(x)f(x)\right) $.}, with $x^{[ A^{k-t}]_{ij}}$ being a concave function and $1/\mu_{0}(B)\int_B d\mu_{0} = 1$, we have that	
		\begin{align*}
		\bar{g}_{B}(x) & \leq \left( \frac{1}{\mu_{0}(B)} \int\limits_{B}p_\theta^j(x) d\mu_{0}(\theta)\right) ^{[ A^{k-t}]_{ij}} .
	\end{align*}
thus,
	\begin{align*}
{\mathbb{P}}_{B_1}\left[\log \frac{\bar g_{B_2}(\boldsymbol{X}^k)}{\bar g_{B_1}(\boldsymbol{X}^k)} \geq y\right] 
	& \leq \exp(-y / 2)\prod_{t=1}^{k}\prod_{j=1}^{n} \rho({P}_{B_1}^{j},{P}_{B_2}^{j})^{[ A^{k-t}]_{ij}},
	\end{align*}
	where ${P}_{B}^{j}$ is the probability distribution associated with the density $ \frac{1}{\mu_{0}(B)} \int\limits_{B}p_\theta^j(x) d\mu_{0}(\theta)$.
	
	Assumption \ref{assum:bound} and the compactness of $\Theta$ guarantees that $\rho({P}_{B_1}^{j},{P}_{B_2}^{j})>\alpha$ for some positive $\alpha$, thus similarly as in Lemma~\ref{lemmab}, we have that
	\begin{align*}
	\mathbb{P}_{B_1}\left[\log \frac{\bar g_{B_2}(\boldsymbol{X}^k)}{\bar g_{B_1}(\boldsymbol{X}^k)} \geq y\right] 
	& \leq \exp(-y / 2)\exp\left(\log\frac{1}{\alpha}\frac{4\log n}{1- \delta} \right) \prod_{t=1}^{k}\prod_{j=1}^{n} \rho({P}_{B_1}^{j},{P}_{B_2}^{j})^{1/n} \\
	& \leq \exp(-y / 2)\exp\left(\log\frac{1}{\alpha}\frac{4\log n}{1- \delta} \right) \exp\left( -\frac{k}{n} \sum_{j=1}^{n}h^2({P}_{B_1}^{j},{P}_{B_2}^{j})\right) .
	\end{align*}
	
	Finally, by using the metric defined for the $n$-Hellinger ball and the fact that for a metric $d(A,B)$ for two sets $A$ and $B$ $d(A,B) = \inf_{x\in A, y\in B} d(x,y)$ we have
	\begin{align*}
	\mathbb{P}_{B_1}\left[\log \frac{\bar g_{B_2}(\boldsymbol{X}^k)}{\bar g_{B_1}(\boldsymbol{X}^k)} \geq y\right] 
	& \leq \exp(-y / 2)\exp\left(\log\frac{1}{\alpha}\frac{4\log n}{1- \delta} \right) \exp\left( -k \left(\sqrt{ \frac{1}{n} \sum_{j=1}^{n}h^2({P}_{B_1}^{j},{P}_{B_2}^{j})}\right) ^2\right)  \\
	& \leq \exp(-y / 2)\exp\left(\log\frac{1}{\alpha}\frac{4\log n}{1- \delta} \right) \exp\left( -k \left(\sqrt{\frac{1}{n} \sum_{i=1}^{n} h^2\left(P_{\theta_1}^j,P_{\theta_2}^j\right) }  -r^1 - r^2\right)^2 \right). 
	\end{align*}
	\qed
\end{proof}

Lemma \ref{lemma2} provides a concentration result for the logarithmic ratio between two weighted densities over a pair of subsets $B_1$ and $B_2$. The terms involving the auxiliary variable $y$ and the influence of the graph, via $\delta$ are the same as in Lemma \ref{lemmab}. Moreover, the rate at which this bound decays exponentially is influenced now by the radius of the two disjoint Hellinger balls where $B_1$ and $B_2$ are contained respectively. 

The bound provided in Lemma \ref{lemma2} is defined for the random variables $\boldsymbol{X}^k$ having a distribution ${ \boldsymbol{P}}_B^{\otimes k}$. Nonetheless, $\boldsymbol{X}^k$ are distributed according to $\boldsymbol{P}^{\otimes k}$. Therefore, we introduce a lemma that relates the Hellinger affinity of distributions defined over subsets of $\Theta$. 
\begin{lemma}\label{rhos_b}
	Let Assumptions \ref{assum:graph}, \ref{assum:bound} and \ref{assum:max} hold. Consider ${ \boldsymbol{P}}_B^{\otimes k}$ as the distribution with density $ g_B$ as defined in Eq. \eqref{density_g}, for $B \subseteq \mathcal{B}_R$. Then $	{{h({\boldsymbol{P}}_{B}^{\otimes k},\boldsymbol{P}^{\otimes k})  \leq  \sqrt{nk}R}}$.
\end{lemma}
\begin{proof}
	By Jensen's inequality we have that
		\begin{align*}
		\sqrt{ g_B(\boldsymbol{x})} \geq \frac{1}{\mu_0(B)}\int_B \sqrt{\prod_{t=1}^k \prod_{j=1}^{n} p^j_\theta(x^j_t)}d\mu_0(\theta).
		\end{align*}
	Then, by definition of the Hellinger affinity, it follows that 
	\begin{align*}
		\rho({\boldsymbol{P}}_{B}^{\otimes k},\boldsymbol{P}^{\otimes k})& \geq \int_{\boldsymbol{\mathcal{X}}^{\otimes k}}\sqrt{\prod_{t=1}^k \prod_{j=1}^{n} p^j(x^j_t)}\left( \frac{1}{\mu_0(B)}\int_B\sqrt{\prod_{t=1}^k \prod_{j=1}^{n} p^j_\theta(x^j_t)}d\mu_0(\theta)\right)  d\lambda^{\otimes nk}(\boldsymbol{x}).
	\end{align*}
	By using the Fubini-Tonelli Theorem, we obtain
	\begin{align*}
			\rho({\boldsymbol{P}}_{B}^{\otimes k},\boldsymbol{P}^{\otimes k}) 
		& \geq \frac{1}{\mu_0(B)}\int_B \int_{\boldsymbol{\mathcal{X}}^{\otimes k}}\sqrt{\prod_{t=1}^k \prod_{j=1}^{n} p^j(x^j_t)}\sqrt{\prod_{t=1}^k \prod_{j=1}^{n} p^j_\theta(x^j_t)}d\lambda^{\otimes nk}(\boldsymbol{x})d\mu_0(\theta)\\
		& = \frac{1}{\mu_0(B)}\int_B  \prod_{t=1}^k \prod_{j=1}^{n}  \rho(P^j,P^j_\theta) d\mu_0(\theta)\\
& = \frac{1}{\mu_0(B)}\int_B  \prod_{t=1}^k \prod_{j=1}^{n} \left( 1 - h^2(P^j,P^j_\theta)\right) d\mu_0(\theta).
\end{align*}
Finally, by the Weierstrass product inequality it follows that
\begin{align*}
	\rho({\boldsymbol{P}}_{B}^{\otimes k},\boldsymbol{P}^{\otimes k}) & \geq \frac{1}{\mu_0(B)}\int_B  \left( 1 - \sum_{t=1}^k \sum_{j=1}^{n} h^2(P^j,P^j_\theta)\right) d\mu_0(\theta)\\
	& = \frac{1}{\mu_0(B)}\int_B  \left( 1 - n\frac{1}{n}\sum_{t=1}^k \sum_{j=1}^{n} h^2(P^j,P^j_\theta)\right) d\mu_0(\theta)\\
	& \geq \frac{1}{\mu_0(B)}\int_B  \left( 1 - nkR^2\right) d\mu_0(\theta),
	\end{align*}
	where the last line follows by the fact that any density $\boldsymbol{P}_{\theta}$, inside the 
	$n$-Hellinger ball defined in the statement of the lemma, is at most at a distance $R$ to $\boldsymbol{P}$. 
	\qed
\end{proof}

Finally, before presenting our main result for compact sets of hypotheses, we will state an assumption regarding the necessary mass all agents should have around the correct hypothesis $\theta^*$ in their initial beliefs.

\begin{assumption}\label{assum:initial}
	The initial beliefs of all agents are equal. Moreover, they have the following property: for any constants ${{C \in (0,1]}}$ and $r \in (0,1]$ there exists a finite positive integer $K$, such that 
	\begin{align*}
		\mu_0\left( \mathcal{B}_{\frac{C}{\sqrt{k}}} \right) \geq \exp\left( -k\frac{r^2}{32}\right) \qquad \text{for all} \ k\geq K.
	\end{align*}
\end{assumption}

Assumption \ref{assum:initial} implies that the initial beliefs should have enough mass around the correct hypothesis $\theta^*$ when we consider balls of small radius. 
Particularly, as we take Hellinger balls of radius decreasing as $O(1/\sqrt{k})$, the corresponding initial beliefs should not decrease faster than $O(\exp(-k))$.

The assumption can almost always be satisfied by taking initial beliefs to be uniform. The reason is that, in any fixed dimension, the volume of a ball of radius $O(1/\sqrt{k})$ will usually scale as a polynomial in $1/\sqrt{k}$, whereas we only need to lower bound it by a decaying exponential in $k$. For concreteness, we show how this assumption is satisfied by an example. 

\smallskip

\noindent {\bf Example:} Consider a single agent, with a uniform initial, belief receiving observations from a standard Gaussian distribution, i.e. $X_k \sim \mathcal{N}(0,1)$. The variance is known and the agent would like to estimate the mean. Thus the models are $P_\theta=\mathcal{N}(\theta,1)$. 
Now,  the Hellinger distance can be explicitly written as
\begin{align*}
h^2(P,P_\theta) & =  1 - \exp\left(-\frac{1}{4}\theta^2\right).
\end{align*}

Therefore, the Hellinger balls of radius $1/\sqrt{k}$ will correspond to euclidean balls in the parameter space of radius
\begin{align*}
 2\sqrt{\log \left( \frac{1}{1-\frac{1}{k}}\right)}.
\end{align*}
Uniform initial belief indicates that $\mu_0\left( \mathcal{B}_{\frac{C}{\sqrt{k}}} \right)  = \Theta(\frac{1}{\sqrt{k}})$, which can be made larger than $\exp(-k\frac{r^2}{32})$ for sufficiently large $k$.

\bigskip

We are ready now to state our main result regarding the concentration of beliefs around $\theta^*$ for compact sets of hypotheses.

\begin{theorem}\label{main_cont}
	Let Assumptions \ref{assum:graph}, \ref{assum:bound}, \ref{assum:max} and \ref{assum:initial} hold, and let ${\sigma\in (0,1)}$ be a given probability tolerance level. Moreover, for any $r \in (0,1]$, let $\{R_k\}$ be a decreasing sequence such that for $k=1,\hdots,$ ${{R_k \leq \min\left\lbrace  \frac{\sigma}{2\sqrt{2kn}},\frac{r}{4}\right\rbrace }}$. Then, the beliefs $\{\mu_k^i\},$ $i\in V,$ generated by the update rule in~Eq.~\eqref{protocol} have the following property: with probability $1-\sigma$,
	\begin{align*}
	\mu^i_{k+1}(\mathcal{B}_r )& \geq  1 -  \chi \exp\left( - \frac{k}{16}r^2\right)   \qquad \text{ for all } i \text{ and all }k\geq \max\{N,K\}  
	\end{align*}
	where  
	\begin{align*}
		N=& \inf \left\lbrace t \geq 1 \Bigg| \exp \left(\log \frac{1}{\alpha} \frac{4 \log n}{1- \delta} \right)  \sum\limits_{l= 1}^{L-1} K_l\exp\left( -\frac{t}{32}  r_{l+1}^2\right)  < \frac{\sigma}{2} \right\rbrace, 
	\end{align*}
	with $K$ as defined in Assumption \ref{assum:initial}, $\chi=\sum\limits_{l =1}^{L-1}\exp(- \frac{1}{16}r_{l+1}^2)$ and ${\delta = 1-\eta/n^2}$, where $\eta$ is the smallest positive element of the matrix $A$. 
\end{theorem}
\begin{proof}\label{proof_cont}
	Lets start by analyzing the evolution of the beliefs on a measurable set $B$ with $\theta^* \in B$. From Eq.~\eqref{protocol} we have that
	\begin{align*}
	\mu^i_k(B) & = \int\limits_{B}  \prod\limits_{t=1}^{k} \prod\limits_{j=1}^{n} p_\theta^j(X_t^j)^{[ A^{k-t}]_{ij}} d\mu_{0}(\theta) \Bigg/\int\limits_{\Theta}  \prod\limits_{t=1}^{k} \prod\limits_{j=1}^{n} p_\theta^j(X_t^j)^{[ A^{k-t}]_{ij}} d\mu_{0}(\theta)  \\
	& \geq 1 - \int\limits_{B^c}  \prod\limits_{t=1}^{k} \prod\limits_{j=1}^{n} p_\theta^j(X_t^j)^{[ A^{k-t}]_{ij}} d\mu_{0}(\theta) \Bigg/\int\limits_{B}  \prod\limits_{t=1}^{k} \prod\limits_{j=1}^{n} p_\theta^j(X_t^j)^{[ A^{k-t}]_{ij}} d\mu_{0}(\theta).
	\end{align*}
	
	Now lets focus specifically on the case where $B$ is a $n$-Hellinger ball of radius $r>0$ with center at $\theta^*$. In addition, since $R_k < r$, we get	
	\begin{align*}
	\mu^i_k(\mathcal{B}_r ) &\geq 1 - \int\limits_{\mathcal{B}_r ^c}  \prod\limits_{t=1}^{k} \prod\limits_{j=1}^{n} p_\theta^j(X_t^j)^{[ A^{k-t}]_{ij}} d\mu_{0}(\theta) \Bigg/\int\limits_{\mathcal{B}_{R_k} }  \prod\limits_{t=1}^{k} \prod\limits_{j=1}^{n} p_\theta^j(X_t^j)^{[ A^{k-t}]_{ij}} d\mu_{0}(\theta).
	\end{align*}
	
	Our goal will be to use the concentration result in Lemma \ref{lemma2}. Thus, we can multiply and divide by $\mu_0(\mathcal{B}_{R_k} )$ to obtain
	\begin{align*}
	\mu^i_k(\mathcal{B}_r ) &\geq 1 - \int\limits_{\mathcal{B}_r ^c}  \prod\limits_{t=1}^{k} \prod\limits_{j=1}^{n} p_\theta^j(X_t^j)^{[ A^{k-t}]_{ij}} d\mu_{0}(\theta) \Bigg/ \bar g_{\mathcal{B}_{R_k} }(\boldsymbol{X}^k) \mu_0(\mathcal{B}_{R_k} )
	\end{align*}
	
	Moreover, we use the covering of the set $\mathcal{B}_r^c $ to obtain,
	\begin{align}\label{bound3}
	\mu^i_k(\mathcal{B}_r ) &\geq 1 -  \sum\limits_{l=1}^{L-1} \sum\limits_{m=1}^{K_l} \int\limits_{\mathcal{F}_{l,m}}\prod\limits_{t=1}^{k} \prod\limits_{j=1}^{n} p_\theta^j(X_t^j)^{[ A^{k-t}]_{ij}} d\mu_{0}(\theta) \Bigg/ \bar g_{\mathcal{B}_{R_k} }(\boldsymbol{X}^k) \mu_0(\mathcal{B}_{R_k} ) \nonumber \\
	&\geq 1 -  \sum\limits_{l=1}^{L-1} \sum\limits_{m=1}^{K_l} \bar g_{\mathcal{F}_{l,m}}(\boldsymbol{X}^k) \mu_0(\mathcal{F}_{l,m}) \Bigg/ \bar g_{\mathcal{B}_{R_k} }(\boldsymbol{X}^k) \mu_0(\mathcal{B}_{R_k} ) .
	\end{align}
	
	The previous relation defines a ratio between two densities, i.e. $\bar g_{\mathcal{F}_{l,m}}(\boldsymbol{X}^k) / \bar g_{\mathcal{B}_{R_k} }(\boldsymbol{X}^k)$, both for the wighted likelihood product of the observations, where the numerator is defined over to the set $\mathcal{F}_{l,m}$ and the denominator with respect to the set $\mathcal{B}_{R_k} $.
	
	Lemma \ref{lemma2} provides a way to bound term  $\bar g_{\mathcal{F}_{l,m}}(\boldsymbol{X}^k) / \bar g_{\mathcal{B}_{R_k} }(\boldsymbol{X}^k)$ with high probability, thus
	\begin{align*}
	& \mathbb{P}_{\mathcal{B}_{R_k} }\left( \left\lbrace \boldsymbol{X}^k \Bigg|\sup_{l,m} \log \frac{\bar g_{\mathcal{F}_{l,m}}(\boldsymbol{X}^k)}{\bar g_{\mathcal{B}_{R_k} }(\boldsymbol{X}^k)} \geq y\right\rbrace \right)\leq \sum\limits_{l = 1}^{L-1}\sum\limits_{m=1}^{K_l}\mathbb{P}_{\mathcal{B}_{R_k} }\left(\log \frac{\bar g_{\mathcal{F}_{l,m}}(\boldsymbol{X}^k)}{\bar g_{\mathcal{B}_{R_k} }(\boldsymbol{X}^k)} \geq y \right) \\
	& \leq \sum\limits_{l= 1}^{L-1}\sum\limits_{m=1}^{K_l}\exp(- y / 2) \exp \left( \log \frac{1}{\alpha} \frac{4 \log n}{1- \delta} \right) \exp\left( -k \left( \sqrt{\frac{1}{n}\sum_{j=1}^{n} h^2(p^j_{m},p^j)} -\delta_l -R_k\right)^2\right) \\
	& \leq \sum\limits_{l = 1}^{L-1}\sum\limits_{m=1}^{K_l}\exp(- y / 2) \exp \left(\log \frac{1}{\alpha}  \frac{4 \log n}{1- \delta} \right) \exp\left( -k \left( r_{l+1} -\delta_l -R_k\right)^2\right).
	\end{align*}
	where $p^j_{m}$ is the density of at the point $\theta = m \in S_{\varepsilon_l}$, where $S_{\varepsilon_l}$ is the maximal $\varepsilon_l$ separated set of $\mathcal{F}_l$ as in Definition \ref{separated}.
	
	Particularly, lets use the covering proposed in \cite{bir15}, where
	$\delta_{l} = r_{l+1}/2$. From this choice of covering, we have that
	\begin{align*}
	r_{l+1} -\delta_l -R_k & > r_{l+1} -r_{l+1}/2 -r_{l+1}/4 \\
	& = r_{l+1}/4
	\end{align*}
	where we have used the assumption that $R_k \leq r/4$ or equivalently $R_k \leq r_{l}/4$ for all $1\leq l \leq L$.

	Thus, we can set $y = -\frac{k}{16}r_{l+1}^2$ and it follows that
	\begin{align}\label{prob_count}
\mathbb{P}_{\mathcal{B}_{R_k} }\left( \left\lbrace \boldsymbol{X}^k \Bigg|\sup_{l,m} \log \frac{\bar g_{\mathcal{F}_{l,m}}(\boldsymbol{X}^k)}{\bar g_{\mathcal{B}_{R_k} }(\boldsymbol{X}^k)} \geq y\right\rbrace \right) & \leq \exp \left(\log \frac{1}{\alpha}  \frac{4 \log n}{1- \delta}\right) \sum\limits_{l=1}^{L-1}K_l \exp\left( -\frac{k}{16}r_{l+1}^2\right) .
	\end{align}
	
The probability measure in Eq. \eqref{prob_count} is computed for $\boldsymbol{X}^k$ distributed according to $\boldsymbol{P}_{\mathcal{B}_{R_k} }^{\otimes k}$. Nonetheless, $\boldsymbol{X}^k$ is distributed according to the (slightly different) $\boldsymbol{P}^{\otimes k}$. Our next step is to relate these two measures. 

First, we have that for any distribution $\boldsymbol{P}_{\theta} \in \mathcal{B}_{R_k} $, from the Definition \ref{h_balls} of the $n$-Hellinger ball, it holds that 
\begin{align*}
\sqrt{\frac{1}{n}\sum_{j=1}^{n}h^2(P^j_{\theta},P^j)} \leq R_k,
\end{align*}
and we relate the total variation distance and the Hellinger affinity as in Lemma $1$ in \cite{lecam73}; for any measurable set $A$ it holds that
		\begin{align*}
		\sup_A \left(  \boldsymbol{P}_{\mathcal{B}_{R_k} }^{\otimes k} (A) - \boldsymbol{P}^{\otimes k} (A)\right) ^2 & \leq 1 - \rho^2(\boldsymbol{P}_{\mathcal{B}_{R_k} }^{\otimes k},\boldsymbol{P}^{\otimes k}) , 
		\end{align*}
and by definition of the Hellinger affinity we have that
		\begin{align*}
		\sup_A \left(  \boldsymbol{P}_{\mathcal{B}_{R_k} }^{\otimes k} (A) - \boldsymbol{P}^{\otimes k} (A)\right) ^2& = 1 - (1-h^2(\boldsymbol{P}_{\mathcal{B}_{R_k} }^{\otimes k},\boldsymbol{P}^{\otimes k}))^2 \\
		& \leq 2h^2(\boldsymbol{P}_{\mathcal{B}_{R_k} }^{\otimes k},\boldsymbol{P}^{\otimes k}),
		\end{align*}
		where first we have used the relation that for any $x\in \mathbb{R}$, it holds that ${{1-(1-x^2)^2<2x^2}}$. Then, from Lemma \ref{rhos_b} we have that
		\begin{align*}
		\sup_A \left(  P_{\mathcal{B}_{R_k} } (A) - \boldsymbol{P}^{\otimes k} (A)\right) ^2 & \leq 2knR_k^2.
		\end{align*}

		Therefore, by considering the measurable subset $\Gamma^k = \left\lbrace \boldsymbol{X}^k \Bigg|\sup_{l,m} \log \frac{\bar g_{\mathcal{F}_{l,m}}(\boldsymbol{X}^k)}{\bar g_{\mathcal{B}_{R_k} }(\boldsymbol{X}^k)} \geq -\frac{k}{16}r_{l+1}^2\right\rbrace $, we have that
		\begin{align*}
		\mathbb{P}\left( \Gamma^k\right)  & < \mathbb{P}_{\mathcal{B}_{R_k} }\left( \Gamma^k\right) + \sqrt{2kn}R_k \\
		& \leq \exp \left( \log \frac{1}{\alpha} \frac{4 \log n}{1- \delta}\right) \sum\limits_{l=1}^{L-1}K_l \exp\left( -\frac{k}{16}r_{l+1}^2\right) +\frac{\sigma}{2}.
		\end{align*}		
	
	Furthermore, we are interested in finding a large enough $k$ such that the probability described in Eq. \eqref{prob_count} is at most $\sigma$. Thus, we define
		\begin{align*}
	N \geq \inf \left\lbrace t \geq 1 \Bigg| \exp \left(\log \frac{1}{\alpha}  \frac{4 \log n}{1- \delta} \right)  \sum\limits_{l= 1}^{L-1}K_l\exp\left( -\frac{t}{16}r_{l+1}^2\right) < \frac{\sigma}{2} \right\rbrace .
	\end{align*}

Moreover, from Eq.~\eqref{bound3} we obtain that with probability $1-\sigma$ for all $k \geq N$,
	\begin{align*}
	\mu^i_k(\mathcal{B}_r ) &\geq 1 -  \sum\limits_{l= 1}^{L-1} \sum\limits_{m=1}^{K_l} \exp\left( -\frac{k}{16}r_{l+1}^2\right) \frac{\mu_0(\mathcal{F}_{l,m}) }{\mu_0(\mathcal{B}_{R_k} ) } \\
	& = 1 -  \sum\limits_{l =1}^{L-1} \exp\left( -\frac{k}{16}r_{l+1}^2\right) \frac{\mu_0(\mathcal{F}_{l}) }{\mu_0(\mathcal{B}_{R_k} ) } \\
	& \geq 1 -  \frac{1}{\mu_0(\mathcal{B}_{R_k} ) }\sum\limits_{l = 1}^{L-1} \exp\left( -\frac{k}{16}r_{l+1}^2\right) .
	\end{align*}
	
	Now, lets define $\chi=\sum\limits_{l = 1}^{L-1}\exp\left( -\frac{1}{16}  r_{l+1}^2\right)$, then it follows that 
	\begin{align*}
	\mu^i_{k}(\mathcal{B}_r )& \geq 1 -   \frac{1}{\mu_0(\mathcal{B}_{R_k} ) }\sum\limits_{l = 1}^{L-1} \exp\left( -\frac{k}{16}  r_{l+1}^2\right)\\
	& = 1 -   \frac{1}{\mu_0(\mathcal{B}_{R_k} ) }\sum\limits_{l = 1}^{L-1} \exp\left( -\frac{1}{16}  r_{l+1}^2\right)\exp\left( -\frac{k-1}{16}  r_{l+1}^2\right)\\
	& \geq 1 -   \frac{1}{\mu_0(\mathcal{B}_{R_k} ) }\chi\exp\left( -\frac{k-1}{16}  r^2\right),
		\end{align*}
	where the last inequality follows from $r_{l}\geq r$ for all $L \leq l \leq 1$. Finally, by Assumption \ref{assum:initial} we have that, for all $k\geq K$
	\begin{align*}
	\mu^i_{k}(\mathcal{B}_r ) & \geq 1 -  \chi \exp(- \frac{k-1}{16}r^2 + \frac{k-1}{32} r^2)\\
	& = 1 -  \chi \exp(- \frac{k-1}{32}r^2),
	\end{align*}
	or equivalently $\mu^i_{k+1}(\mathcal{B}_r )  \geq 1 -  \chi \exp(- \frac{k}{32}r^2)$.
	\qed
\end{proof}

Analogous to Theorem \ref{main_count}, Theorem \ref{main_cont} provides a probabilistic concentration result for the agents' beliefs around a Hellinger ball of radius $r$ with center at $\theta^*$ for sufficiently large $k$. 

\section{Conclusions}\label{sec_conclusion}
We have proposed an algorithm for distributed learning with both countable and compact sets of hypotheses. Our algorithm may be viewed as a distributed version of Stochastic Mirror Descent applied to the problem of minimizing the sum of Kullback-Leibler divergences. Our results show non-asymptotic geometric convergence rates for the beliefs concentration around the true hypothesis.

It would be interesting to explore how variations on stochastic approximation algorithms will produce new non-Bayesian update rules for more general problems. Promising directions include acceleration results for proximal methods, other Bregman distances or constraints within the space of probability distributions. 

Furthermore we have modeled interactions between agents as exchanges of local  probability distributions (i.e., beliefs) between neighboring nodes in a graph. An interesting open question is to understand to what extent this can be reduced when agents transmit only an approximate summary of their beliefs. We anticipate that future work will additionally consider the effect of parametric approximations allowing nodes to communicate only a finite number of parameters coming from, say, Gaussian Mixture Models or Particle Filters.

%
\bibliographystyle{spbasic} 
\bibliography{IEEEfull,bayes_cons_3}

\end{document}